\documentclass[12pt]{amsart}
\usepackage[utf8]{inputenc}
\usepackage{amssymb,latexsym}
\usepackage{enumerate}
\usepackage{hyperref}
\usepackage{enumitem}
\usepackage{amssymb}
\usepackage{amsthm}
\usepackage{amsmath}
\usepackage{color}
\usepackage[margin=1.5in]{geometry}

\makeatletter
\@namedef{subjclassname@2010}{\textup{2010} Mathematics Subject Classification}
\makeatother

\newtheorem{thm}{Theorem}[section]
\newtheorem{cor}[thm]{Corollary}
\newtheorem{prob}[thm]{Problem}
\newtheorem{lem}[thm]{Lemma}
\newtheorem{prop}[thm]{Proposition}
\theoremstyle{definition}
\newtheorem{defn}[thm]{Definition}
\theoremstyle{remark}
\newtheorem{rem}[thm]{Remark}

\numberwithin{equation}{section}

\begin{document}
\author   [Y. El Khatiri]{Youssef El Khatiri}
\address{Department of Mathematics
	\\ Mathematical Research Center of Rabat
	\\ Laboratory of Mathematics, Statistics and Applications
	\\ Faculty of Sciences \\ Mohammed V University in Rabat  \\ Rabat, Morocco.}

\email{\textcolor[rgb]{0.00,0.00,0.84}{ elkhatiriyoussef@hotmail.com\\
		youssef\_elkhatiri2@um5.ac.ma}}

\title{Order isomorphisms  in $C^*$-algebras}

\maketitle

\begin{abstract}
We provide a complete description of the order isomorphisms between the self-adjoint parts of $C^*$-algebras. Furthermore, we characterize such isomorphisms between general operator intervals in $AW^*$-algebras. For the description, we use Jordan $^*$-isomorphisms and closed operators in the regular rings of  $AW^{*}$-algebras. This work generalizes previous results on von Neumann algebras.  
\end{abstract}

\section{Introduction}

A map $\Phi: (\mathcal{F},\leq_{\mathcal{F}})\to (\mathcal{G},\leq_{\mathcal{G}})$ between two partially ordered sets   is said to be  an order isomorphism if it is a bijection that preserves the order in both directions, i.e., $$a \leq_{\mathcal{F}}b\iff \Phi(a)\leq_{\mathcal{G}} \Phi(b),\qquad a,b\in \mathcal{F}.$$

The study of order isomorphisms has received considerable attention from many authors. 
For instance, Connes classified linear order isomorphisms between noncommutative $L^2$-spaces; see~\cite{Connes}. 
Also, Arendt and De Cannière studied these isomorphisms in the  setting of Fourier algebras, demonstrating that such isomorphisms characterize the underlying locally compact groups; see~\cite{ Arendt}. In this paper, we focus on studying these isomorphisms in the context of $C^*$-algebras, for several reasons. 
 $C^*$-algebras provide a rigorous mathematical framework for describing the noncommutative structures underlying the formulation of quantum mechanics. 
The pivotal idea rests on the Gelfand--Naimark theorem, which establishes that every $C^*$-algebra can be represented as an algebra of bounded operators on a Hilbert space. 
In this framework, the observables of a quantum system correspond to the self-adjoint elements of a $C^*$-algebra, while the states are given by positive linear functionals on this algebra; see, e.g., \cite{B3}. Additionally, classifying order isomorphisms between operator intervals in $C^*$-algebras remains a highly challenging problem within the theory of preservers.

One of the fundamental results on order isomorphisms in the context of $C^*$-algebras is due to Kadison. In \cite{Kadison1} and  \cite{Kadison2}, he classified the linear isometric maps between $C^*$-algebras.
 We note that there is a strong connection between these maps and order isomorphisms.  Notably, any isometry becomes a linear order isomorphism after a simple transformation.  This connection is also apparent in the main results of both papers, which can be summarized as follows: 
\begin{thm}(Kadison)	
Any  unital linear order isomorphism between the self-adjoint parts of $C^*$-algebras extends uniquely to a Jordan $^*$-isomorphism between the corresponding $C^*$-algebras.
\end{thm}
 It is natural to inquire whether we can describe these  isomorphisms, without the linearity assumption, using Jordan $^*$-isomorphisms as well. If not, what is the form of these isomorphisms?

 When the linearity assumption is removed from Kadison’s theorem, the study of order isomorphisms on general $C^*$-algebras becomes considerably more challenging, requiring the development of a new approach.
 This became apparent in subsequent papers  that address order isomorphisms without the linearity assumption. These papers focused on specific $C^*$-algebras possessing rich structural properties, such as commutative algebras, those with minimal projections, algebras exhibiting spatial properties.  
 
 For the commutative case, in \cite{sanchez}, Cabello S\'anchez provided a complete description of all order isomorphisms between the self-adjoint parts of commutative $C^*$-algebras. Subsequently, Ehsani presented a complete characterization of all order isomorphisms between the effect algebras of these $C^*$-algebras; see \cite{Ehsani}. These examples are merely illustrative among several related results in the literature. For the noncommutative case, Molnár provided a characterization of order automorphisms between the self-adjoint parts of the full algebra of bounded operators on a Hilbert space; see \cite{molnar}.  In the same direction, \v{S}emrl characterized order isomorphisms between all operator intervals in \cite{Semrl0, Semrl}. Subsequently, Mori established in \cite{mori} a complete description of order isomorphisms between operator intervals in general von Neumann algebras. These results generalize  several earlier works by Molnár and \v{S}emrl concerning type~$\mathrm{I}$ factors. A parallel generalization of the works of Molnár and \v{S}emrl in the setting of algebras of continuous operator-valued functions on compact spaces was obtained by Abdelali and El~Khatiri; see \cite{Our}.  

This raises considerable interest in the overall structure of order isomorphisms between operator intervals in arbitrary $C^*$-algebras.  Mori posed the following question in \cite[Section 5]{mori}.
\begin{prob}\label{problem1}
Let $\mathcal{A}$ and $\mathcal{B}$ be two $C^*$-algebras. Suppose that $\Phi$ is an order isomorphism between their self-adjoint parts. Under what conditions must $\Phi$ be an affine map?
\end{prob}

The problem arises naturally, since in most von Neumann algebras, precisely, those without a type $\mathrm{I}_1$ direct summand, the corresponding isomorphisms are linear. The scope of this paper is  to provide a complete description for such isomorphisms.  Our strategy is to first determine the forms of the order isomorphisms between operator intervals in a class of $C^*$-algebras known as $AW^*$-algebras.

A $C^*$-algebra $\mathcal{A}$ is called an $AW^*$-algebra if, for every non-empty subset $\mathcal{F} \subseteq \mathcal{A}$, there exists a projection $p \in \mathcal{A}$ satisfying 
$$
\{a \in \mathcal{A} : ba = 0, \text{ for all } b \in \mathcal{F}\} =\{pa: \ a\in\mathcal{A}\}.
$$
The notion of $AW^*$-algebras originates from von Neumann’s seminal investigations into operator rings, now referred to as von Neumann algebras. In an effort to extract the purely algebraic core of this theory, Kaplansky formulated the concept of $AW^*$-algebras as an abstract generalization; see \cite{Kaplansky}. Every von Neumann algebra is  an $AW^*$-algebra, but the reverse implication does not hold. A classic  example of an $AW^*$-algebra that is not a von Neumann algebra can be constructed by taking the quotient of the algebra of all bounded, complex-valued, Borel-measurable functions on $\mathbb{R}$,  the field of real numbers by the ideal of functions $f$ whose support set $$\{t\in \mathbb{R}: f(t)\neq0 \}$$  is meagre in $\mathbb{R}$; see \cite{Saito-wright}. Although the theory of $AW^*$-algebras continues to present numerous open questions, substantial advances have been achieved in recent years.  For further details on these algebras, the reader is referred to \cite{Berberian, Saito-wright}.

To address the case of effect algebras of $AW^*$-algebras without a type $\mathrm{I}_2$ direct summand, we rely essentially on Dye’s theorem concerning orthoisomorphisms; see \cite[Theorem 4.3]{Hamhalter}. As is usual in preserver problems, algebras of type~$\mathrm{I}_2$ require special treatment; see, e.g., \cite{Our2, El Khatiri, mori}. In this context, we make use of the results in \cite{Our}, which provides a characterization of order isomorphisms between effect algebras of $C^*$-algebras consisting of continuous operator-valued functions on a compact Hausdorff space. In order to obtain a complete description of all order isomorphisms between effect algebras in $AW^*$-algebras, it is necessary to employ larger structures, namely, the rings of measurable operators; see \cite{Berberian, Berberian1, Berberian2}. Our proof for the remaining cases of operator intervals is self-contained and does not rely on a larger framework. Consequently, although our approach shares certain similarities with that of \cite{mori}, the overall argument and techniques employed here are substantially different.

After studying all order isomorphisms between the self-adjoint parts of $AW^*$-algebras, we turn to the general case of $C^*$-algebras. Our approach to describing these isomorphisms involves reducing the problem from arbitrary $C^*$-algebras to the case of monotone complete $C^*$-algebras. We achieve this by extending order isomorphisms between the self-adjoint parts of $C^*$-algebras to order isomorphisms between the self-adjoint parts of their regular monotone completions. Our tools are  Wright's theorem on the Dedekind completion of  Archimedean partially ordered vector spaces with an order unit and Hamana's theorem on the regular monotone completion of $C^*$-algebras; see \cite{Hamana, Wright}.

The rest of the paper is structured as follows. Section \ref{S1}   will be devoted to preliminary results. Section \ref{S2} is devoted to  stating our principal results. It is well known that one of the cornerstones in classifying order isomorphisms between operator intervals is determining the form of order isomorphisms between effect algebras. For this reason, we begin Section \ref{S3} by describing order isomorphisms between the effect algebras of $AW^*$-algebras as the first step in analysing operator intervals in $AW^*$-algebras. In Section \ref{S4}, we provide a complete description of order isomorphisms between the remaining types of operator intervals. Section \ref{S5} reduces Problem \ref{problem1} to the case of monotone complete $C^*$-algebras. Finally, in Section \ref{S6}, we present several remarks illustrating the non-uniqueness in the description of order isomorphisms between effect algebras. The paper concludes with two open questions that suggest promising directions for future research.

\section{Preliminaries}\label{S1}

 Throughout this paper, every $C^*$-algebra $\mathcal{A}$ is assumed to be unital with unit 1, and its norm is denoted by $\| \cdot\|$. The self-adjoint part of $\mathcal{A}$, denoted by $\mathcal{A}^{sa}$, forms a real Banach space and, equipped with the Jordan product $$a \circ b = \tfrac{1}{2}(ab + ba),\qquad a,b\in \mathcal{A},$$ it constitutes a Jordan algebra. The positive cone of $\mathcal{A}$, denoted by $\mathcal{A}^+$, is defined as $\{a^2 : a \in \mathcal{A}^{sa}\}$. This cone induces a natural partial order on the Jordan algebra   $\mathcal{A}^{sa}$. Specifically, for two self-adjoint elements $a, b \in \mathcal{A}^{sa}$, we say that $a \leq b$ if $b - a \in \mathcal{A}^+$. The set of all positive invertible elements in $\mathcal{A}$ is denoted by $\mathcal{A}^{++}$. We write $a<b$ if $b-a\in\mathcal{A}^{++}$. We define the operator interval $[a,b]_{\mathcal{A}} := \{\, c \in \mathcal{A}^{sa} : a \leq c \leq b \,\}$.  
 In particular, the effect algebra of $\mathcal{A}$, denoted by $\mathcal{A}^{1}$, is given by $[0,1]_{\mathcal{A}}$.

  We will use also the following notation. 
 \begin{itemize}
 	\item  $\mathcal{P(A)}:=\{p\in\mathcal{A}:\ p^*=p=p^2\}$ and $p^\perp=1-p$, for all $p\in \mathcal{P(A)}$.
 	\item  For $\mathcal{F}\subseteq\mathcal{A}$, $p\in \mathcal{P(A)}$,  $p\mathcal{F}:=\{pa: a\in \mathcal{F}\}$ and  $\mathcal{F}p:=\{ap: a\in \mathcal{F}\}$.
 	\item For any $n\in \mathbb{N}^*$, $\mathrm{M}_n(\mathcal{A})$ denotes the algebra of all $n\times n$ matrices over $\mathcal{A}$.
 \end{itemize}
We denote  Murray-von Neumann equivalence by the symbol $\sim$. Specifically, for $p,q\in\mathcal{P(A)}$, we write $p \sim q$ when there exists a partial isometry $v\in\mathcal{A}$ such that $vv^* = p$ and $v^*v = q$. Additionally, we write $p \preceq q$ when  $p\sim q'$ for some $q'\in\mathcal{P(A)}$ such that $q'\leq q$.
 
 Let  $\mathcal{I}$ be  subset of $\mathcal{A}$. For a compact Hausdorff space $\mathcal{X}$, denote by $\mathcal{C}(\mathcal{X}, \mathcal{I})$ the set of all continuous $\mathcal{I}$-valued functions on $\mathcal{X}$. For brevity, we shall write $\mathcal{C}(\mathcal{X})$ instead of $\mathcal{C}(\mathcal{X}, \mathbb{C})$. Then $\mathcal{C}\big(\mathcal{X}, [0,1]\big)=\mathcal{C}(\mathcal{X})^{1}$, $\mathcal{C}\big(\mathcal{X}, [0,\infty)\big)=\mathcal{C}(\mathcal{X})^+$ and $\mathcal{C}(\mathcal{X}, \mathbb{R})=\mathcal{C}(\mathcal{X})^{sa}$.
 
  Let $\mathcal{H}$ be a complex Hilbert space, we will  denote by $\mathcal{B}(\mathcal{H})$  the set of all bounded linear operators on $\mathcal{H}$. 
 
 We recall  that a Jordan $^*$-isomorphism  $J: \mathcal{A}\to \mathcal{B}$ between two   $C^*$-algebras  is a linear bijection  that preserves the adjoint and  the Jordan product, i.e., $$J(a^*)=J(a)^* \quad \mbox{and} \quad J(a \circ b)=J(a)\circ J(b),\qquad a,b\in \mathcal{A}.$$
Therefore,  the restriction of any such map    $J$ to a  subset $\mathcal{F}$ of $\mathcal{A}^{sa}$  is an order isomorphism, with respect to the natural orders of $\mathcal{A}$ and $\mathcal{B}$, from  $\mathcal{F}$ onto $J(\mathcal{F})$.

\subsection{Operator intervals in  $C^*$-algebras}
Let $\mathcal{A}$ be a $C^{*}$-algebra.  
A subset $\mathcal{F} \subseteq \mathcal{A}$ is called an operator interval if  
$\mathcal{F} = \mathcal{A}^{sa}$, or there exists an $a\in \mathcal{A}^{sa}$,  
such that $\mathcal{F}$ is one of the sets listed below:
$$\{c\in \mathcal{A}^{sa}:  c\geq a\},$$
$$\{c\in \mathcal{A}^{sa}:  c\leq a\},$$
$$\{c\in \mathcal{A}^{sa}: a<c \},$$
$$\{c\in \mathcal{A}^{sa}:  c<a \},$$
or there exist $a,b \in \mathcal{A}^{sa}$, with $a < b$,  
such that $\mathcal{F}=[a,b]_{\mathcal{A}}$, or $\mathcal{F}$ is one of the sets listed below:
$$\{c\in \mathcal{A}^{sa}: a\leq c< b\},$$
$$\{c\in \mathcal{A}^{sa}: a< c\leq b\},$$
$$\{c\in \mathcal{A}^{sa}: a< c< b\}.$$

Two partially ordered sets  $ (\mathcal{F},\leq_{\mathcal{F}})$   and $(\mathcal{G}, \leq_{\mathcal{G}})$  are said to be order anti-isomorphic if there exists    a bijection  $\Phi: (\mathcal{F},\leq_{\mathcal{F}})\to (\mathcal{G},\leq_{\mathcal{G}})$  that inverts  order in both directions, i.e., $$a \leq_{\mathcal{F}}b\iff \Phi(b)\leq_{\mathcal{G}} \Phi(a), \qquad a,b\in \mathcal{F}.$$
It is well known that the problem of characterizing the general forms of order isomorphisms between operator intervals can be reduced to  special cases. Namely,  using simple  transformations, we can verify  that each of these intervals is either order isomorphic or  anti-isomorphic to one of the following subsets: $$\mathcal{A}^1,  \mathcal{A}^+,  \mathcal{A}^{++} \mbox{ and }  \mathcal{A}^{sa}.$$
We refer the reader to  \cite{Semrl} for a proof, where the corresponding results for the $C^*$-algebras of all bounded operators on  Hilbert space  are proven. The same proof also holds here for a general $C^*$-algebra. Consequently, to obtain the general forms of order isomorphisms between operator intervals, it suffices to consider order isomorphisms between sets of the following the aforementioned forms.
 If $\mathcal{A}$ is commutative, then $\mathcal{A}^{sa}$ and $\mathcal{A}^{++}$ are easily seen to be order isomorphic. However, this is not possible when $\mathcal{A}$ is non-commutative, as will be discussed later. In the absence of commutativity, no further reduction in the number of intervals that must be studied is possible. This follows from the fact that each of these intervals is not order isomorphic to the others.
\subsection{ Regular monotone completion  of   $C^*$-algebras}\label{2.2}

An extension of a $C^*$-algebra $\mathcal{A}$ is a pair $(\mathcal{B}, J)$, where $\mathcal{B}$ is a $C^*$-algebra and $J: \mathcal{A} \hookrightarrow \mathcal{B}$ is a unital $^*$-monomorphism.

On the set of all extensions of $\mathcal{A}$, we can define an  equivalence relation $\sim$ as follows: $(\mathcal{B}, J) \sim (\mathcal{C}, J')$ if there exists a unital $^*$-isomorphism $J'': \mathcal{B} \to \mathcal{C}$ such that $J'' \circ J = J'$.

We will identify $\mathcal{A}$ with its image $J(\mathcal{A}) \subseteq \mathcal{B}$ and abbreviate $(\mathcal{B}, J)$ to $\mathcal{B}$.   An extension $\mathcal{B}$ of $\mathcal{A}$ is called regular if for each element $b \in \mathcal{B}^{sa}$, $b$ is the least upper-bound in $\mathcal{B}^{sa}$ of the set $\{a\in\mathcal{A}^{sa} : a\leq b\}$. For now, we will be content with this definition of regularity; a deeper treatment of this concept will be provided in Section \ref{S3}. 

Let $\mathcal{F}$ be a partially ordered set and $\mathcal{G}$ a subset. We use $\sup_{\mathcal{F}}  \mathcal{G}$ to denote the least upper-bound (supremum) of $\mathcal{G}$ in  $\mathcal{F}$, if it exists. Similarly, $\inf_{\mathcal{F}} \mathcal{G}$ denotes the greatest lower-bound (infimum), if it exists.

A subset $\mathcal{F}$ of $\mathcal{A}^{sa}$ is monotone closed in $\mathcal{A}^{sa}$ if, for any bounded  increasing (resp. decreasing)  net $\mathcal{G}$ in $\mathcal{F}$ with  least upper-bound $\sup_{\mathcal{A}^{sa}}\mathcal{G}$ (resp. greatest
lower-bound  $\inf_{\mathcal{A}^{sa}}\mathcal{G}$),
then $\sup_{\mathcal{A}^{sa}}\mathcal{G}$ 
(resp. $\inf_{\mathcal{A}^{sa}}\mathcal{G}$) is in $\mathcal{F}$. The monotone closure of $\mathcal{F}$ in $\mathcal{A}^{sa}$ is the intersection of all monotone closed subsets of $\mathcal{A}^{sa}$ that contain $\mathcal{F}$.

In \cite[Theorem 3.1]{Hamana}, Hamana proved that any $C^{*}$-algebra $\mathcal{A}$ possesses  a 
regular extension  $\overline{\mathcal{A}}$, where $\overline{\mathcal{A}}$ is a monotone complete $C^*$-algebra satisfying  the following properties:

\begin{enumerate}
	\item[(i)]  $ \overline{\mathcal{A}}^{sa}$  is the  monotone closure of $\mathcal{A}^{sa}$  in $ \overline{\mathcal{A}}^{sa}$. \\
	
	\item[(ii)] For each $a\in  \overline{\mathcal{A}}^{sa}$, we have   $a=\sup_{\overline{\mathcal{A}}^{sa}}\{c\in \mathcal{A}^{sa}: c\leq a\}$. \\
	
	\item[(iii)] If  $\mathcal{F}\subseteq \mathcal{A}^{sa}$ such that $\sup_{\mathcal{A}^{sa}}\mathcal{F}=a$. Then, $a= \sup_{\overline{\mathcal{A}}^{sa}}\mathcal{F}$.
\end{enumerate} 
He proved also that the algebra $\overline{\mathcal{A}}$  is unique up to the equivalence relation $\sim$,  and it is called 
a regular monotone completion of $\mathcal{A}$. In this paper, $\overline{\mathcal{A}}$ denotes a fixed   regular monotone completion of $\mathcal{A}$. While this extension is not unique, all regular monotone completions  of  $\mathcal{A}$ are  mutually $^*$-isomorphic, and in particular, they are  mutually Jordan $^*$-isomorphic. Since these isomorphisms preserve the order structure of $C^*$-algebras, we can choose any one of these extensions without introducing any conflict.

\subsection{Regular ring of  finite $AW^*$-algebras} It is well known that  the classical type theory of von Neumann algebras can  be extended to $AW^*$-algebras; see \cite{Berberian}, \cite{Saito-wright}.  Throughout this paper, in all matters concerning $AW^*$-algebras, we follow the terminology of Berberian \cite{Berberian}. If $\mathcal{A}$ is an $AW^*$-algebra,  for each strictly positive integer $n$ there exists a unique central projection $p_\mathcal{A}(n)$  such that $p_\mathcal{A}(n)\mathcal{A}:=\mathcal{A}_n$ is either $0$ or of type~$\mathrm{I}_n$ and  $p_{\mathcal{A}}(n)^{\perp}\mathcal{A}:=\mathcal{A}_{-n}$ does not have a type~$\mathrm{I}_n$ direct summand. For more information see \cite[Theorems~3, p. 94; 4, p. 116]{Berberian}.
 
Let  $\mathcal{A}$ be an  $AW^*$-algebra, for  any  $a\in \mathcal{A}^{sa}$, there exists a smallest projection $   \mathrm{RP}(a)\in \mathcal{P(A)}$  that satisfies $$  \mathrm{RP}(a)a=a  \mathrm{RP}(a)=a.$$

Let $\mathcal{A}$ be a finite $AW^*$-algebra. An operator with closure is a pair of sequences $(a_n, p_n)$, where $(a_n) \subseteq \mathcal{A}$ and $(p_n)\subseteq\mathcal{P(A)}$ is an increasing sequence  with  $\mathrm{sup}_{\mathcal{P(A)}}\{p_n: n\}=1$, satisfying $a_n p_m = a_m p_m$ and $a_n^* p_m = a_m^* p_m$, for all $m \leq n$. Two operators with closure $(a_n,p_n)$, $(b_n,p'_n)$ are said to be equivalent if there exists a sequence of projections $(q_n)$ with $\mathrm{sup}_{\mathcal{P(A)}}\{q_n: n\}=1$ such that $a_n q_n = b_n q_n$.
The closed operator $[a_n,p_n]$ represents the class of all operators with closure that are equivalent to $(a_n,p_n)$. 

By \cite[Theorem 2.1]{Berberian1}, the set \(\mathcal{R(A)}\) of all closed operators affiliated with $\mathcal{A}$ is an associative algebra over the complex numbers, with involution \(*\), with respect to the operations:
$$[a_n, p_n] + [b_n, q_n] = [a_n + b_n,  \mathrm{inf}_{\mathcal{P(A)}}\{p_n,q_n\}]$$
$$
\lambda [a_n, p_n] = [\lambda a_n, p_n]
$$
$$
[a_n, p_n]^* = [a_n^*, p_n]$$
$$[a_n, p_n][b_n, q_n] = [a_n b_n, p'_n]$$
where $(p'_n)$ is an increasing sequence of projections such that $\sup_{\mathcal{P(A)}}\{p'_n: n\}=1$. The mapping $a \mapsto [a, 1]$ ($a \in \mathcal{A}$) is a \(*\)-isomorphism of $\mathcal{A}$ into $\mathcal{R(A)}$, and $[1, 1]$ is the  unity element of $\mathcal{R(A)}$.

For notational simplicity, we identify each element $a\in\mathcal{A}$ with $[a,1]\in\mathcal{R(A)}$. Note that $\mathcal{R(A)}$ has no more bounded element than $\mathcal{A}$, in the sense that, if $[a_n,p_n]\in \mathcal{R(A)}$ such that the sequence $(\|a_n\|)$ is bounded, then $[a_n,p_n]=[a,1]$ for some $a\in \mathcal{A}$; see \cite[Theorem 5.1]{Berberian1}. On the other hand, $\mathcal{R(A)}$ has more invertible elements. Indeed, if $a \in \mathcal{A}^{sa}$ with $ \mathrm{RP}(a) = 1$, then $a$ is invertible in $\mathcal{R(A)}$ but $a$ is not necessarily invertible in $\mathcal{A}$, for more details; see \cite[Corollary 7.5]{Berberian1}.  We denote the inverse of $a$ in $\mathcal{R(A)}$ by $a^{-1}$.

The following lemma provides an order characterization of elements in $\mathcal{A}^1$ that are invertible in $\mathcal{R(A)}$.
\begin{lem}\label{inverse}
	Let $\mathcal{A}$ be a finite $AW^*$-algebra. Suppose that $a,b\in\mathcal{A}^1$ commute with each other. The element $b-a$ is positive invertible in $\mathcal{R(A)}$ if, and only if,  $a$ is not greater than or equal to any projection in $\mathcal{P(A)}$,  and   for every $c\in\mathcal{A}^1\setminus\{a\}$ with $c\geq a$, there exists a lower-bound $d\in\mathcal{A}^1$ of $\{b,c\}$  such that $d\neq a$.  
\end{lem}
\begin{proof}
	First, we observe that the condition  for every $c\in\mathcal{A}^1\setminus\{a\}$ with $c\geq a$, there exists a lower-bound $d\in\mathcal{A}^1$ of $\{b,c\}$  such that $d\neq a$ is equivalent to the condition that     for every $c\in\mathcal{A}^1\setminus\{a\}$ with $c\geq a$, the set  $\{b-a,c-a\}$ has a non-zero lower-bound in $\mathcal{A}^1$.

	If $ \mathrm{RP}(b-a)=1$, $b-a$ is invertible in $\mathcal{R(A)}$. Otherwise, $p:= \mathrm{RP}(b-a)^\perp\in\{b-a\}''$  is a  non-zero projection; see \cite[Corollary 1, p. 17]{Berberian}. Since $a$ and $b$ commute with each other, we derive that $\{a,b\}\subseteq\{a,b\}'$, hence $\{a,b\}''\subseteq\{a,b\}'$, which implies that $\{a,b\}''$ is commutative. On the other hand $\{a,b\}'\subseteq\{b-a\}'$,  thus $\{b-a\}''\subseteq\{a,b\}''$, which implies that $p$ commutes with $a$ and $b$. So, $$\mathrm{inf}_{\mathcal{A}^1}\{a,p\}=\mathrm{inf}_{\mathcal{A}^1}\{b,p\}=ap=bp.$$
	Denote $c=ap^\perp+p$. Thus $c\neq a $  and $c\geq a$. We easily see that $$\mathrm{inf}_{\mathcal{A}^1}\{b-a,c-a\}=0.$$  
	This proves the if assertion.
	
	Conversely,  let $c\in\mathcal{A}^1\setminus\{a\}$ with $c\geq a$. We can write $(b-a)^{-1/2}(c-a)(b-a)^{-1/2}=[d'_n,p_n]$ with $(d'_n,p_n)$  lying in the cone of a commutative $AW^*$-subalgebra $\mathcal{C(X)}$ of $\mathcal{A}$, and  $(d'_n)$ is non-zero  increasing sequence; see \cite[Theorem 4.2]{Berberian1}. Put $d_n=\mathrm{inf}_{\mathcal{C(X},\mathbb{R})}\{1,d_n'\}$,  since $(\|d_n\|)$ is a non-zero bounded sequence, there exists a non-zero element $d\in\mathcal{A}^+$ such that $d=[d_n,p_n]$. So $d \leq 1$ and $d\leq(b-a)^{-1/2}(c-a)(b-a)^{-1/2}$. Hence, $c_0:=(b-a)^{1/2}d(b-a)^{1/2}$ is a non-zero lower-bound of $\{b-a,c-a\}$  in $\mathcal{A}^1$.	
\end{proof}

\section{The statement of main results}\label{S2}

The following theorem summarizes some of the principal results of this paper, theorems \ref{Jordan}, \ref{A+}, \ref{A++}, \ref{A^sa}, \ref{++non=sa}, which collectively provide a characterization of order isomorphisms between operator intervals in $AW^*$-algebras without a type~$\mathrm{I}_1$ direct summand.

\begin{thm}\label{intervals}
	Let $\mathcal{A}$ and $\mathcal{B}$ be two $AW^*$-algebras such that  $\mathcal{A}$ does not have a type~$\mathrm{I}_1$ direct summand, then  the following statements hold.
	\begin{enumerate}
		\item  Let $\Phi: \mathcal{A}^1\to\mathcal{B}^1$  be an order isomorphism with $\Phi(1/2)=1/2$. Then, $\Phi$ extends uniquely  to a Jordan $^*$-isomorphism from $\mathcal{A}$ onto $\mathcal{B}$. 	
		
		\item  Let $\Phi: \mathcal{A}^+\to\mathcal{B}^+$  be an order isomorphism. Then there exist an element $b\in\mathcal{B}^{++}$ and a Jordan $^*$-isomorphism $J:\mathcal{A}\to\mathcal{B}$ such that $$\Phi(a)=bJ(a)b, \qquad a\in \mathcal{A}^+.$$ 
		
		\item  Let $\Phi: \mathcal{A}^{++}\to\mathcal{B}^{++}$  be an order isomorphism. Then there exist an  element $b\in\mathcal{B}^{++}$ and a Jordan $^*$-isomorphism $J:\mathcal{A}\to\mathcal{B}$ such that $$\Phi(a)=bJ(a)b,\qquad  a\in \mathcal{A}^{++}.$$

	\item Let $\Phi: \mathcal{A}^{sa}\to\mathcal{B}^{sa}$  be an order isomorphism. Then there exist an element $b\in\mathcal{B}^{++}$, an element $c\in\mathcal{B}^{sa}$  and a Jordan $^*$-isomorphism $J:\mathcal{A}\to\mathcal{B}$  such that $$\Phi(a)=bJ(a)b+c,\qquad \ a\in \mathcal{A}^{sa}.$$
	
\item The operator intervals $\mathcal{A}^{sa}$  and $\mathcal{B}^{++}$  are not order isomorphic.
	\end{enumerate}
\end{thm}

\begin{rem}
	For effect algebras,  we actually  provide  more general descriptions in Theorem \ref{general_formula}.  For  finite $AW^*$-algebras, we provide a complete description  using Jordan $^*$-isomorphisms and  closed operators of the regular rings in the sense of Berberian  \cite{Berberian1}.	
\end{rem}

While order isomorphisms of operator intervals in the commutative setting (the $\mathrm{I}_1$ case) have been extensively investigated; see, e.g., \cite{Lochan, Marovt}, the existing results do not provide a comprehensive treatment for all commutative $C^*$-algebras nor for all types of operator intervals. In the following proposition, we address this gap by examining each type of operator interval within general commutative $C^*$-algebras.

\begin{prop}\label{comm_case}
	Let $\mathcal{X}$ and $\mathcal{Y}$ be two Hausdorff compact spaces and $\mathcal{I}\in\{[0,1], [0,\infty), \mathbb{R}  \}$.	Suppose that $\Phi:\mathcal{C}(\mathcal{X},\mathcal{I})\to \mathcal{C}(\mathcal{Y},\mathcal{I})$  is  an order isomorphism.  Then there exist a homeomorphism $\mu:\mathcal{Y}\to\mathcal{X}$, a  dense $G_\delta$-subset $\mathcal{Y}_{0}$ of $\mathcal{Y}$ and a family of increasing homeomorphisms $(f_y)_{y\in\mathcal{Y}_{0}}$ of $\mathcal{I}$ such that 
	\begin{equation}\label{comm-form}
		\Phi(f)(y)=f_y\big( f(\mu(y))\big),\qquad f\in \mathcal{C}(\mathcal{X},\mathcal{I}), y\in\mathcal{Y}_{0}.	
	\end{equation}
\end{prop}

The following theorem demonstrates that the study of order isomorphisms between operator intervals in $AW^*$-algebras should be conducted separately for two distinct classes of such algebras: the commutative ones and those without a commutative direct summand. 

\begin{thm}\label{I_1}
	Let $\mathcal{A}$ and  $\mathcal{B}$ be two  $AW^*$-algebras. Suppose that $\mathcal{I}\subseteq\mathcal{A}$ and $\mathcal{J}\subseteq\mathcal{B}$ be two operator intervals. Let $\Phi: \mathcal{I}\to \mathcal{J}$ be an order isomorphism. Then    there exists a unique  pair of   order isomorphisms $\Phi_1:p_\mathcal{A}(1)\mathcal{I}\to p_\mathcal{B}(1)\mathcal{J}$ and  $\Phi_{-1}:p_\mathcal{A}(1)^\perp \mathcal{I}\to p_\mathcal{B}(1)^\perp\mathcal{J}$  such that  $$\Phi(a)=\Phi_1\left( p_\mathcal{A}(1)a\right) +\Phi_{-1}\left( p_\mathcal{A}(1)^{\perp}a\right) , \qquad a\in \mathcal{I}.$$ 
\end{thm}

\begin{rem}
	We conclude that the preceding results together provide a complete characterization of all order isomorphisms between operator intervals in $AW^*$-algebras, except for the effect algebras.
\end{rem}

The following theorem shows that the problem of characterizing order isomorphisms between the self-adjoint parts of general $C^*$-algebras can be reduced to the case of monotone complete $C^*$-algebras, and subsequently, to the case of $AW^*$-algebras.

\begin{thm}\label{extension}
Let $\mathcal{A}$ and  $\mathcal{B}$ be two  $C^*$-algebras. Suppose that $\Phi: \mathcal{A}^{sa}\to \mathcal{B}^{sa}$ is an order isomorphism. Then  there exists  a  unique  order isomorphism  $\overline{\Phi}:{\overline{\mathcal{A}}}^{sa}\to \overline{\mathcal{B}}^{sa}$ extends $\Phi$.  
\end{thm}

Let $\mathcal{A}$ be a $C^{*}$-algebra. It is straightforward to verify that $\overline{\mathcal{A}}$ admits the direct-sum decomposition
$$
\overline{\mathcal{A}}=\overline{\mathcal{A}}_{1}\oplus \overline{\mathcal{A}}_{-1}.
$$
Since $\overline{\mathcal{A}}_{1}$ is a commutative $AW^{*}$-algebra, there exists a Stonean space $\mathcal{X}$ such that
$$
\overline{\mathcal{A}}=\mathcal{C}(\mathcal{X})\oplus \overline{\mathcal{A}}_{-1}.
$$
Hence,  for each $a\in\mathcal{A}^{sa}$, there exists a unique $f\in\mathcal{C}(\mathcal{X})$ and $a'\in\overline{\mathcal{A}}_{-1}$ such that $$a=f+a'.$$
Namely, $f=p_\mathcal{A}(1)a$ and $a'=p_\mathcal{A}(1)^\perp a$.

 We conclude the general form of order isomorphisms between the self-adjoint parts of arbitrary $C^*$-algebras.

\begin{cor}\label{main-cor}
	Let $\mathcal{A}$ and $\mathcal{B}$ be two $C^*$-algebras, and let $\mathcal{X}$ and $\mathcal{Y}$ be two compact Hausdorff spaces such that 
	$$
	\overline{\mathcal{A}} = \mathcal{C}(\mathcal{X}) \oplus \overline{\mathcal{A}}_{-1}
	\quad \text{and} \quad
	\overline{\mathcal{B}} = \mathcal{C}(\mathcal{Y}) \oplus \overline{\mathcal{B}}_{-1}.
	$$
	Suppose that $\Phi: \mathcal{A}^{sa} \to \mathcal{B}^{sa}$ is an order isomorphism and  $\overline{\Phi}:{\overline{\mathcal{A}}}^{sa}\to \overline{\mathcal{B}}^{sa}$ its extension given by Theorem \ref{extension}. Then, for every $f+a\in\mathcal{A}^{sa}$, we have 
	$$
	\Phi(f + a) = \overline{\Phi}(f) + \overline{\Phi}(a).
	$$
Moreover, there exist a homeomorphism $\mu: \mathcal{Y} \to \mathcal{X}$, a dense $G_\delta$-subset $\mathcal{Y}_{0}$ of $\mathcal{Y}$, a family of increasing homeomorphisms $(f_y)_{y \in \mathcal{Y}_{0}}$ of $\mathbb{R}$, an element $b \in  (\overline{\mathcal{B}}_{-1})^{++}$, an element $c \in  \overline{\mathcal{B}}_{-1}^{sa}$, and a Jordan $^*$-isomorphism 
	$$J: \overline{\mathcal{A}}_{-1} \to \overline{\mathcal{B}}_{-1}$$
	such that 
	$$
	\Phi(f)(y) = f_y\big(f(\mu(y))\big), 
	\qquad 
	f \in \mathcal{C}(\mathcal{X}), \ y \in \mathcal{Y}_{0},
	$$
	and 
	$$
	\Phi(a) = b J(a) b + c, 
	\qquad 
	a \in \overline{\mathcal{A}}_{-1}.
	$$
\end{cor}

We conclude this section by answering Problem~\ref{problem1}.
\begin{cor}
	Let $\mathcal{A}$ and $\mathcal{B}$ be $C^*$-algebras, and let $\Phi:\mathcal{A}^{sa}\to\mathcal{B}^{sa}$ be an order isomorphism. If the regular monotone completion of either $\mathcal{A}$ or $\mathcal{B}$ has no direct summand of type~$\mathrm{I}_1$, then $\Phi$ is affine.
\end{cor}

\section{Order isomorphisms between effect
algebras}\label{S3}

This section investigates order isomorphisms between effect algebras of $AW^*$-algebras. 

The orthogonal  projections play an important role in  studying order isomorphisms between  effect algebras, see for instance \cite{Our, mori, Semrl}. We recall    that any order isomorphism between effect  algebras of $C^*$-algebras  preserves the orthogonal projections. This fact can be proven by the same method as in \cite[Lemma 3.1]{mori}, so we omit its  proof.  

The following lemma  is   key to analysing the order on projections.

\begin{lem}\label{Homo}
Let $\mathcal{A}$ be a $C^*$-algebra. Let $t\in[0,1]$ and $\alpha$ be a complex number with   $|\alpha|=\sqrt{t(1-t)}$. Suppose that $p=\begin{pmatrix}
	1 & 0 \\ 
	0 & 0
\end{pmatrix}$ and $q=\begin{pmatrix}
	t                & \alpha\\
	\alpha^* & 1-t
\end{pmatrix}$ are in $\mathrm{M}_2(\mathcal{A})$. Then $$\frac{1}{2-t}p=\mathrm{inf}_{\mathrm{M}_2(\mathcal{A})^1}\left\lbrace p,\mathrm{sup}_{\mathrm{M}_2(\mathcal{A})^1}\{q,1/2\}\right\rbrace.$$
\end{lem}
\begin{proof}
For simplicity, we suppose that $\alpha=\sqrt{t(1-t)}$.  We have $$\begin{array}{llllllllllllllll}
	\mathrm{sup}_{\mathrm{M}_2(\mathcal{A})^1}\{q,1/2\}-\dfrac{1}{2-t}p &=& q+1/2q^\perp-\dfrac{1}{2-t}p \\
	\\
	&=&1/2 \begin{pmatrix}
		\dfrac{t(1-t)}{2-t} & \sqrt{t(1-t)}\\ \\
		\sqrt{t(1-t)} & 2-t
	\end{pmatrix}\\ \\
	&=& 1/2 \begin{pmatrix}
		\sqrt{\dfrac{t(1-t)}{2-t}} & 0\\ \\
		\sqrt{2-t} &  0
	\end{pmatrix}  \begin{pmatrix}
		\sqrt{\dfrac{t(1-t)}{2-t}} & 0\\ \\
		\sqrt{2-t} &  0
	\end{pmatrix}^* \\ 
	&\geq& 0.
\end{array}$$	
Hence, $\dfrac{1}{2-t}p$ is a lower-bound of $\left\lbrace p,\mathrm{sup}_{\mathrm{M}_2(\mathcal{A})^1}\{q,1/2\}\right\rbrace $. On the other hand, for every  $a=\begin{pmatrix}
	a_{11}&a_{12}\\
	a_{21} & a_{22}
\end{pmatrix}\in \mathrm{M}_2(\mathcal{A})^+$ that is majorized by $p$, $a_{12}=a_{21}=a_{22}=0$ and $a_{11}\in\mathcal{A}^+$. If $a\leq q+1/2q^\perp$, then   $$\begin{array}{llllllll}
	a&=&\begin{pmatrix}
		1 & 0 \\ \\
		\dfrac{-\sqrt{t(1-t)}}{2-t} & 0
	\end{pmatrix}^*a\begin{pmatrix}
		1 & 0 \\ \\
		\dfrac{-\sqrt{t(1-t)}}{2-t} & 0
	\end{pmatrix}\\\\
	&\leq& \begin{pmatrix}
		1 & 0 \\ \\
		\dfrac{-\sqrt{t(1-t)}}{2-t} & 0
	\end{pmatrix}^*  (q+1/2q^\perp) \begin{pmatrix}
		1 & 0 \\ \\
		\dfrac{-\sqrt{t(1-t)}}{2-t} & 0
	\end{pmatrix}\\\\
	&=& \dfrac{1}{2-t}p.
\end{array}$$	
This completes the proof.
\end{proof}

It is known that the set of all orthogonal projections of an $AW^*$-algebra forms a lattice.  The following lemma   is motivated by \cite[Proposition 3.9]{Our}. It provides an order characterization of the orthogonality between  projections.

\begin{lem}\label{orth}
Let $\mathcal{A}$ be an $AW^*$-algebra.  Two projections $p,q\in\mathcal{P(A)}$    are orthogonal if and only if
$$\mathrm{inf}_{\mathcal{A}^1}\{p,1/2\}=\mathrm{inf}_{\mathcal{A}^1}\left\lbrace p, \mathrm{sup}_{\mathcal{A}^1}\{q, 1/2\}\right\rbrace.$$ 

\end{lem}
\begin{proof}
We can easily check that $1/2p=\mathrm{inf}_{\mathcal{A}^1}\{p,1/2\}$ is a lower-bound of $\left\lbrace p, \mathrm{sup}_{\mathcal{A}^1}\{q, 1/2\}\right\rbrace$. On the other hand, if $p,q$ are orthogonal, using the fact that $p\leq q^{\perp}$, for every lower-bound $a\in\mathcal{A}^1$  of $\left\lbrace p, \mathrm{sup}_{\mathcal{A}^1}\{q, 1/2\}\right\rbrace$, we have  $$a=pap\leq p(q+1/2q^{\perp})p= 1/2pq^\perp p=1/2p.$$

Conversely,  by considering $p$  and $q$ as projections of the $AW^*$-subalgebra    $\mathrm{sup}_{\mathcal{P(A)}}\{p,q\}\mathcal{A}\mathrm{sup}_{\mathcal{P(A)}}\{p,q\}$,   then without loss of generality we can  
assume that $\mathrm{sup}_{\mathcal{P(A)}}\{p,q\}=1$. Denote $p_1=\mathrm{inf}_{\mathcal{P(A)}}\{p,q^\perp\}$, $p_2=\mathrm{inf}_{\mathcal{P(A)}}\{p^\perp,q\}$ and $p_3=\mathrm{sup}_{\mathcal{P(A)}}\{p-p_1,q-p_2\}$.  It is rather easy to
verify that $\mathrm{inf}_{\mathcal{P(A)}}\{p,q\}=0$, the projections $p_i$, $i\in\{1,2,3\}$, form a partition of identity, i.e., $p_1+p_2+p_3=1$ and they  commute with $p$ and $q$. Note that $p, q$ are orthogonal if and only if $p_3p, p_3q$ are orthogonal. The projections $p_3p,p_3q$ are in position $(\mathrm{p})$ as projections of the $AW^*$-algebra $p_3\mathcal{A}p_3$, i.e.,
$$
\begin{array}{lllll}
	\mathrm{inf}_{\mathcal{P}(p_3\mathcal{A}p_3)}\{p_3p,p_3q\}&=&\mathrm{inf}_{\mathcal{P}(p_3\mathcal{A}p_3)}\{p_3p^\perp,p_3q\}
	=\mathrm{inf}_{\mathcal{P}(p_3\mathcal{A}p_3)}\{p_3p,p_3q^\perp\}\\
	&=&\mathrm{inf}_{\mathcal{P}(p_3\mathcal{A}p_3)}\{p_3p^\perp,p_3q^\perp\}=0.
\end{array}
$$ Then by \cite[Proposition 2.4, Proposition 2.5]{Hamhalter}, there exists an $AW^*$-subalgebra of $p_3\mathcal{A}p_3$ contains
$p_3p, p_3q$ and 
isomorphic to $\mathcal{C(X},\mathrm{M}_2(\mathbb{C}))$, where $\mathcal{X}$ is a compact  Hausdorff space.  Moreover,  we can
represent $p_3p$ and $p_3q$  as matrix valued functions $$(p_3p)(x)=\begin{pmatrix}
	1 & 0\\
	0 & 0
\end{pmatrix},\qquad   (p_3q)(x)=\begin{pmatrix}
	f(x) & \sqrt{f(x)(1-f(x))}\\
	\sqrt{f(x)(1-f(x))} & 1-f(x)
\end{pmatrix}$$ 
on $\mathcal{X}$, for some continuous function  $f$  on $\mathcal{X}$ with values in $[0, 1]$.   By Lemma \ref{Homo}, $$\dfrac{1}{2-f}p_3p\leq p_3(q+1/2q^\perp).$$
Thus, $$\begin{array}{llllllll}
1/2p+(\dfrac{1}{2-f}-1/2)p_3p  &=& \dfrac{1}{2-f}p_3p +1/2(p_1+p_2)p \\ \\
&\leq&   p_3(q+1/2q^\perp) + 1/2(p_1+p_2)p \\\\
&=& p_3(q+1/2q^\perp) + 1/2p_1 \\ \\
&\leq& p_3(q+1/2q^\perp) + (q+1/2q^\perp)(p_1+p_2)\\ \\
&=& q+1/2q^\perp.
\end{array}$$
 So by  the assumption,  $\dfrac{1}{2-f}=1/2$. Thus, $f=0$, which means that  $p_3p$ and $p_3q$ are orthogonal. 
\end{proof}

\begin{cor}\label{orthoisomorphism}
Let $\mathcal{A}$ and $\mathcal{B}$ be two $AW^*$-algebras.  Suppose that    $\Phi:\mathcal{A}^1\to \mathcal{B}^1$ is an order isomorphism  with $\Phi(1/2)=1/2$. Then the restriction of $\Phi$ to $\mathcal{P(A)}$ is an orthoisomorphism 	onto $\mathcal{P(B)}$, i.e., $\Phi$
maps $\mathcal{P(A)}$ onto $\mathcal{P(B)}$ bijectively and   for every  $p,q \in \mathcal{P(A)}$ $$pq=0 \qquad  \mbox{ if and only if}  \qquad \Phi(p) \Phi(q)=0.$$
\end{cor}

The commutant $\mathcal{F}^{'}$ of a non-empty set $\mathcal{F} \subseteq \mathcal{A}$ is the set of all elements of $\mathcal{A}$ that commute with every element of $\mathcal{F}$. If $\mathcal{A}$ is an $AW^*$-algebra and $a\in\mathcal{A}^{sa}$, by \cite{Saito-wright}, $\{a\}^{''}\cap \mathcal{A}^{sa}$ is  monotone closed. Furthermore, by \cite{Berberian}, $\{a\}^{''}$  can be identified with $\mathcal{C(X)}$, for some Stonean space $\mathcal{X}$, i.e., a compact, Hausdorff space that is extremely disconnected, specifically, the closure of every open  subset is open.

The modulus of  complex number $z$ is denoted by $|z|$.
\begin{lem}\label{finite.sum}
	Let $\mathcal{A}$ be an $AW^*$-algebra and  $a\in\mathcal{A}^1$. Then there exists a sequence $(a_n)\subset\mathcal{A}^1$ such that $$a=\mathrm{sup}_{\mathcal{A}^{1}}\{a_n:n\};$$
and for each $n$, there exist a finite sequence $\{p_1,p_2,...,p_k\}\subseteq\mathcal{P}(\{a\}^{''})$ of  mutually orthogonal projections, and $\{t_1,t_2,...,t_k\}\subset(0,1]$ such that $a_n= \underset{i=1}{\overset{k}{\sum}} t_ip_i$.
	\end{lem}

\begin{proof}
	
It is enough to prove that for each $n$, there exist a finite sequence  of  mutually orthogonal projections $\{p_1,p_2,...,p_k\}\subseteq\mathcal{P}(\{a\}^{''})$, and $\{t_1,t_2,...,t_k\}\subset(0,1]$ such that $$0 \leq a- \sum_{i=1}^{k} t_ip_i \leq 1/n. $$	
	
 	We have $\{a\}^{''} = \mathcal{C(X)}$, where $\mathcal{X}$ is a Stonean space. Fix $n$ and suppose that $x\in\mathcal{X}$. Let $\mathcal{X}_x$ be the closure of the open set $\{y\in\mathcal{X}: |a(y)-a(x)|<\frac{1}{2n}\}$. Thus, $\{\mathcal{X}_x: x\in \mathcal{X}\}$ is a clopen cover of $\mathcal{X}$. Since $\mathcal{X}$ is a compact space, there exists a finite family $\{x_1,x_2,...,x_k\}\subseteq \mathcal{X}$ such that $\{\mathcal{X}_{x_1},\mathcal{X}_{x_2},...,\mathcal{X}_{x_k}\}$  forms a clopen cover of $\mathcal{X}$. Without loss of generality, we may suppose  that $\mathcal{X}_{x_i}$  are mutually disjoint. For each $i\in\{1,2,...,k\}$,  we denote $t_i:=\min\{a(y):y\in\mathcal{X}_{x_i}\}$ and  by   $p_i$  the characteristic function on $\mathcal{X}_{x_i}$. Then $0\leq a(x)- \underset{i=1}{\overset{k}{\sum}} t_ip_i(x) \leq 1/n$, for all $x\in\mathcal{X}$, as desired.
\end{proof}

As is common in preserver problems,  the $\mathrm{I}_2$ case is  special, and
it needs some additional work in comparison with other cases. In our situation, we rely heavily  on  \cite[Theorem 2.4]{Our}. To recall this theorem we need to define what we called   the property $\left(\mathrm{K}_{\mathcal{E}}\right)$. Let  $\mathcal{X}$ be a compact  Hausdorff  space. We say that $\mathcal{X}$  satisfies the property $\left(\mathrm{K}_{\mathcal{E}}\right)$ if every order automorphism on $\mathcal{C}(\mathcal{X},[0,1])$  is continuous, or equivalently, the formula \eqref{comm-form} holds for the entire space $\mathcal{X}$.

\begin{thm}\label{Abdelali.ElKhatiri}
(Abdelali,  El Khatiri) Let $\mathcal{H}$ be a Hilbert space with $\dim \mathcal{H}\geq 2$ and let  $\mathcal{X}$ be a Hausdorff compact space that satisfies  the property $\left(\mathrm{K}_{\mathcal{E}}\right)$. Suppose that $\Phi$ is an order automorphism   on $\mathcal{C}\left( \mathcal{X},\mathcal{B(H)}\right)^{1}$ such that $\Phi(1/2)=1/2$. Then $\Phi$    extends  uniquely  to a Jordan $^*$-isomorphism  on  $\mathcal{C}\left( \mathcal{X},\mathcal{B}(\mathcal{H})\right)$.
\end{thm}

The authors  used the  property $\left(\mathrm{K}_{\mathcal{E}}\right)$ to guarantee that for any faithful (i.e., its central cover is $1$)  abelian projection $p \in \mathcal{P}\left(\mathcal{C}(\mathcal{X},\mathcal{B}(\mathcal{H}))\right)$, the restriction of $\Phi$ to $\{a \in \mathcal{C}(\mathcal{X},\mathcal{B}(\mathcal{H}))^1: 0 \leq a \leq p\}$ is continuous, see \cite[Section 4.1]{Our}. So, if $\mathcal{X}$ does not necessarily satisfy the property $\left(\mathrm{K}_{\mathcal{E}}\right)$, and we could prove, in some way, that the restriction of $\Phi$ to $\{a \in \mathcal{C}(\mathcal{X},\mathcal{B(H)})^1: 0 \leq a \leq p\}$ is continuous for all faithful  abelian projection $p$ in  $\mathcal{P}\left(\mathcal{C}(\mathcal{X},\mathcal{B}(\mathcal{H}))\right)$, then by repeating exactly the same argument, we can prove that $\Phi$ extends  uniquely to a  Jordan $^*$-isomorphism on $\mathcal{C}\left(\mathcal{X}, \mathcal{B}(\mathcal{H})\right)$.

We are now ready to prove one of the main results of this section. However, before we do so, let us recall some notation that will be useful for shortening the proofs. For any order isomorphism $\Phi: \mathcal{A}^1 \to \mathcal{B}^1$, we define the order isomorphism $\Phi^{\perp}: \mathcal{A}^1 \to \mathcal{B}^1$ by
$$a \mapsto 1 - \Phi(1-a), \qquad a\in \mathcal{A}^1.$$

\begin{thm}\label{Jordan}
Let $\mathcal{A}$ and $\mathcal{B}$ be two $AW^*$-algebras. Assume that $\mathcal{A}$ does not have a type~$\mathrm{I}_1$ direct summand. Suppose that $\Phi:\mathcal{A}^1\to\mathcal{B}^1$ is an order isomorphism with $\Phi(1/2)=1/2$. Then $\Phi$ extends  uniquely to a  Jordan $^*$-isomorphism from $\mathcal{A}$ onto $\mathcal{B}$.   
\end{thm}  
\begin{proof}
Notice  that  a projection $p\in \mathcal{P(A)}$ is central if, and only if, it has a unique complement, i.e., there exists a unique projection  $p'\in\mathcal{P(A)}$ such that $\mathrm{inf}_{\mathcal{P(A)}}\{p,p'\}=0$ and $\mathrm{sup}_{\mathcal{P(A)}}\{p,p'\}=1$; see \cite[Theorem 6.6]{Kaplansky}. Notice that   any projection  $p\in\mathcal{P(A)}$ is abelian if, and only if, all projections   $p'\leq p$ are central in $p\mathcal{A}p$, i.e., there exists a unique projection $q'$ such that $\mathrm{sup}_\mathcal{P(A)}\{p',q'\}=p$ and $\mathrm{inf}_\mathcal{P(A)}\{p',q'\}=0$. 

For each  positive integer $n$,  $p_\mathcal{A}(n)$	 is  the  largest  central projection    such that there exist $n$  mutually orthogonal, abelian projections $p_1,p_2,...,p_n$     having  $p_\mathcal{A}(n)$ as a common central cover, i.e., for each $i$, $p_\mathcal{A}(n)$ is the least  upper-bound    central projection such that $p_\mathcal{A}(n)\geq p_i$,  or equivalently, $p_i\sim p_j,$ for all $i,j$, see \cite[Proposition 1, p. 33;   Remark 5, Proposition 1, p. 111 ]{Berberian}, such that $$\mathrm{sup}_\mathcal{P(A)}\{p_1,p_2,...,p_n\}=p.$$

Since $\Phi$ is an order isomorphism preserves the orthogonality between the projections according to Corollary \ref{orthoisomorphism}, we obtain that $\Phi(p_\mathcal{A}(n))=p_\mathcal{B}(n)$, for all positive integer $n$. 

Since $\mathcal{A}$ does not have  an $\mathrm{I}_1$ direct summand, from the above discussion,  $\mathcal{B}$ also does not  have an $\mathrm{I}_1$ direct summand and $\Phi(p_\mathcal{A}(2))=p_\mathcal{B}(2) $.  Following arguments similar to those used in the proof of \cite[Lemma 3.4]{mori} or in the proof of Theorem \ref{I_1} in the next section, we  can show that the maps $p_\mathcal{A}(2)\mathcal{A}^1\to p_\mathcal{B}(2)\mathcal{B}^1$ and $p_\mathcal{A}(2)^{\perp}\mathcal{A}^1\to p_\mathcal{B}(2)^\perp\mathcal{B}^1$ given by $p_\mathcal{A}(2)a\mapsto p_\mathcal{B}(2)\Phi(a)$ and $p_\mathcal{A}(2)^\perp a\mapsto p_\mathcal{B}(2)^\perp\Phi(a)$, respectively, are order isomorphisms. Thus,  to prove our theorem,  it  suffices to treat each of the following  two cases separately: when $\mathcal{A}$ and $\mathcal{B}$ do not have $\mathrm{I}_1$ nor $\mathrm{I}_2$ direct summands, and when $\mathcal{A}$ and $\mathcal{B}$ are of type~$\mathrm{I}_2$.	

We first  consider the case when $\mathcal{A}$ and $\mathcal{B}$ do not have $\mathrm{I}_1$ nor $\mathrm{I}_2$ direct summands. By  Corollary \ref{orthoisomorphism} and \cite[Theorem 4.3]{Hamhalter}, there exists a Jordan $^*$-isomorphism $J:\mathcal{A}\to\mathcal{B}$ such that    $\Phi(p)=J(p)$, for all $p\in\mathcal{P(A)}$.  Denote $\Psi:=J^{-1}\circ \Phi:\mathcal{A}^1\to\mathcal{A}^1$. It is easy to see that $\Psi$ is an order automorphism  with $$\Psi(a)=a, \qquad  a\in\mathcal{P(A)}\cup \{1/2\}.$$	
We will prove that $\Psi$  	is the identity map on $\mathcal{A}^1$.  From \cite[Lemma 4.12]{Kaplansky} and \cite[Theorem 4, p. 116]{Berberian} we conclude the existence of a mutually orthogonal projections $\{p_i:i\in \mathcal{I}\}$ such that $\mathrm{sup}_{\mathcal{P(A)}}\{p_i:i\in \mathcal{I}\}=1$ and $p_i\preceq p_i^\perp$, for all $i\in \mathcal{I}$. Let $p\in   \{p_i:i\in \mathcal{I}\}$, then $p\sim p'$, for some projection $p'\leq p^\perp$. Thus we can identify $(p+p')\mathcal{A}(p+p')$ with $\mathrm{M}_2(p\mathcal{A}p)$ such that  $$p=\begin{pmatrix}
	1 & 0\\
	0 & 0
\end{pmatrix};$$
see \cite[Proposition 1, p. 98]{Berberian}.
Let $t\in[1/2,1]$. By Lemma \ref{Homo}, we have  $tp=\mathrm{inf}_{\mathcal{A}^1}\{p,q+1/2q^\perp\}$ where $$q=1/t\begin{pmatrix}
	2t-1 & 	\sqrt{(2t-1)(  1-t)}\\
	\sqrt{(2t-1)( 1-t)} & 1-t
\end{pmatrix}.$$
Hence, $\Psi(tp)=tp$.  Since  $\Psi^\perp$  and $\Psi$ have the same  properties we obtain that for every $t\in[0,1/2]$ $$\begin{array}{lllll}
	(1-t)p	&=&\Psi^\perp((1-t)p)\\\\
	&=& 1-\Psi \left(p^\perp+tp\right) \\ \\
	&=& 1-\Psi\left( \mathrm{sup}_{\mathcal{A}^1}\left\lbrace p^{\perp},tp\right\rbrace \right) \\ \\
	&=&1-\mathrm{sup}_{\mathcal{A}^1}\left\lbrace \Psi(p^{\perp}),\Psi(tp) \right\rbrace  \\ \\
	&=&p-\Psi(tp).
\end{array}$$
Thus $\Psi(tp)=tp$, for all $t\in[0,1]$ and all $p\in \{p_i:i\in \mathcal{I}\}$. So $\Psi(t)=\Psi\left( \mathrm{sup}_{\mathcal{A}^1}\{tp_i:i\in \mathcal{I}\}\right)=t $, for all $t\in [0,1]$ and then $\Psi(tp)=tp$, for all $t\in[0,1]$ and all $p\in \mathcal{P(A)}$. 

We will prove by induction that $\Psi\left( \underset{i=1}{\overset{n}{\sum}} t_ip_i\right)= \underset{i=1}{\overset{n}{\sum}} t_ip_i$, for all  of all finite family of mutually orthogonal projections  $\{p_1,p_2,...,p_n\}$ in $\mathcal{A}$ that forms a partition of identity, and for all coefficients  $0\leq t_1<t_2...<t_n\leq1$.  This is true
for $n=1$.  Suppose that we have established it for $n\geq 1$, and let $a=\Psi\left( \underset{i=1}{\overset{n+1}{\sum}} t_ip_i\right)$. Using the induction hypothesis, we have $\underset{i=1}{\overset{n-1}{\sum}} t_ip_i\leq a\leq \underset{i=1}{\overset{n-1}{\sum}} t_ip_i+ p_n+p_{n+1}$. By \cite[Lemma 3.2]{Semrl},  $a=\underset{i=1}{\overset{n-1}{\sum}} t_ip_i+a(p_n+p_{n+1})$. Observe that  $t_{n+1}p_{n+1}\leq a\leq t_{n+1}$. Similarly, we deduce  that $ap_{n+1}=t_{n+1}p_{n+1}$. Moreover, $a$ commutes with $p_n$, we deduce that $t_np_n=\inf_{\mathcal{A}^1}\{p_n,a\}=ap_n$. Hence, $a=\underset{i=1}{\overset{n+1}{\sum}} t_ip_i$.
By Lemma \ref{finite.sum}, we derive that $a\leq \Psi(a)$, for all $a\in\mathcal{A}^1$. Since $\Psi$ and $\Psi^{-1}$ have the same properties we conclude that   $\Psi$ is the identity map on the whole effect algebra $\mathcal{A}^1$.

Now, consider the case when $\mathcal{A}$ and $\mathcal{B}$ are of type~$\mathrm{I}_2$.  We can identify  $\mathcal{A}$ with $\mathcal{C(X},\mathrm{M}_2(\mathbb{C}))$ and $\mathcal{B}$ with $\mathcal{C(Y},\mathrm{M}_2(\mathbb{C}))$, where $\mathcal{X}$ and $\mathcal{Y}$ are two Stonean spaces. Since $\mathcal{X}$ and $\mathcal{Y}$ are homeomorphic, we can suppose without loss of generality that $\mathcal{X}=\mathcal{Y}$; see \cite[Theorem 2.4 and  Section 5]{Our}.  For any  faithful abelian projection $q$ in $\mathcal{C(X},\mathrm{M}_2(\mathbb{C}))$ there exists  $f_q\in\mathcal{C(X},\mathbb{C})^{1}$ and $g_q\in \mathcal{C(X},\mathbb{C})$ with  $|g_q(x)|=\sqrt{f_q(x)(1-f_q(x))}$, for all $x\in\mathcal{X}$ such that $$q(x)=\begin{pmatrix}
	f_q(x) & g_q(x)\\
	g_q(x)^* & 1-f_q(x)
\end{pmatrix}, \qquad x\in\mathcal{X}.$$
By \cite[Theorem 2.4]{Our}, $\Phi$ preserves the faithful  abelian projections in both directions. Let $p$ be a faithful   abelian projection in $\mathcal{C(X},\mathrm{M}_2(\mathbb{C}))$. Without loss of generality,   we may represent  $p$ and $\Phi(p)$   by the function  $$x\mapsto \begin{pmatrix}
	1 & 	0\\
	0 & 0
\end{pmatrix}, \qquad x\in\mathcal{X}.$$  
By  Lemma \ref{Homo}, we have  $$\begin{array}{lll}
	2/3p=\mathrm{inf}_{\mathcal{A}^1}\{p,q+1/2q^\perp\}&\iff& \mathrm{inf}_{\mathcal{A}^1}\{p,q+1/2q^\perp\}=\mathrm{inf}_{\mathcal{A}^1}\{p,q^\perp+1/2q\} \\\\
	&\iff& 2/3p^\perp=\mathrm{inf}_{\mathcal{A}^1}\{p^\perp,q+1/2q^\perp\}\\\\
	&\iff& f_q=1/2.
\end{array}$$
Recall that $\Phi(p)=p$, thus $\Phi(p^{\perp})=p^{\perp}$. One can  take  $q$ the projection   defined by $$q(x):= 1/2\begin{pmatrix}
	1 &  1\\
	1 & 1
\end{pmatrix}, \qquad x\in\mathcal{X},$$  to derive that $\Phi(2/3p)=2/3p$ and $\Phi(2/3p^\perp)=2/3p^\perp$.  Observe that $p+2/3p^\perp=\sup_{\mathcal{A}^1}\{p,2/3p^\perp\}$, then, $\Phi(p+2/3p^\perp)=p+2/3p^\perp$. By the same way we prove that $\Phi(p^\perp+2/3p)=p^\perp+2/3p$. Note  that $2/3$ is the unique element in $\mathcal{A}^1$ that satisfies $2/3p\leq 2/3\leq p^\perp+2/3p$ and  $2/3p^\perp\leq 2/3\leq p+2/3p^\perp$.  Consequently, $\Phi(2/3)=2/3$. Since $\Phi^\perp$ and $\Phi$ have the same properties, we conclude that $\Phi(1/3)=1-\Phi^\perp(2/3)=1/3$.

To prove that $\Phi(t) = t$ for all $t \in [0,1]$, we adapt the methodology outlined in \cite[Proof of Theorem 3.8]{mori}. By applying the reasoning from the preceding discussion to the restrictions of $\Phi$ on the subintervals $[0, 2/3]_{\mathcal{A}}$ and $[1/3, 1]_{\mathcal{A}}$, and iterating this process indefinitely, we demonstrate that $\Phi(t) = t$ holds for every $t$ in a dense subset of $[0,1]$.  Since $\Phi$  is an order isomorphism, we obtain that \begin{equation}\label{Phi(t)=t}
	\Phi(t)=t,  \qquad t\in[0,1]
\end{equation}

Let $p\in\mathcal{C(X},\mathrm{M}_2(\mathbb{C}))$ be a faithful abelian projection, the map $\Phi_p$ defined on $\mathcal{C(X},[0,1])$ by $$\Phi_p(r)(x):=\|\Phi(rp)(x)\|, \qquad x\in\mathcal{X}, r\in\mathcal{C(X},[0,1]),$$
is an order automorphism; see \cite[the proof of Corollary 3.7]{Our}. By \eqref{Phi(t)=t}, we conclude that $\Phi_p(t)=t$, for all $t\in[0,1]$. By Proposition \ref{comm_case},  there exists a homeomorphism $\mu_p$ of $\mathcal{X}$ such that $\Phi_p(r)(x)=r(\mu_p(x))$, for every $x$ in dense subset of $\mathcal{X}$. Hence, $\Phi_p(r)=r\circ\mu_p$. Therefore, the maps $\{\Phi_p: p \mbox{ faithful abelian  projection}\}$ are continuous. By Theorem \ref{Abdelali.ElKhatiri}, $\Phi$ is extendible to a to unique Jordan $^*$-isomorphism on $\mathcal{C(X},\mathrm{M}_2(\mathbb{C}))$.  This completes
the proof.
\end{proof}

In order to address a more general case, we will use   results from the  noncommutative integration theory. Specifically, we will depend  on Berberian's results   on the regular ring of  finite $AW^*$-algebra; see \cite{Berberian, Berberian1}.

Let $\mathcal{A}$ be an $AW^*$-algebra (resp. finite  $AW^*$-algebra). Let $T\in\mathcal{A}$ (resp. in $\mathcal{R(A)}$)  be a positive invertible element. Using \cite[Lemma 8.12]{Berberian2}, we conclude that the map $\Phi_T$ defined  by	
\begin{equation}\label{Phi_T}
\Phi_T(a):=(1+T^{-2})^{1/2}\left(  1-(1+TaT)^{-1}\right) (1+T^{-2})^{1/2},	
\end{equation}
for all $a\in\mathcal{A}^1$, is an order automorphism on $\mathcal{A}^1$. By a simple computation we obtain  that 
\begin{equation}\label{inverse_Phi_T}
\Phi_T^{-1}(a)=T^{-1}\left(\left(   1-(1+T^{-2})^{-1/2}a(1+T^{-2})^{-1/2}\right)^{-1} -1\right) T^{-1},	
\end{equation}
for all $a\in\mathcal{A}^1$. In particular, for every   $\alpha\in[0,1]$, we have
\begin{equation}\label{Phi(alpha)}
\Phi_T(\alpha)=\alpha (1+T^{2})(1+\alpha T^2)^{-1} \quad \mbox{and} \quad \Phi_T^{-1}(\alpha)=\alpha \left( (1-\alpha)T^2+1\right)^{-1}. 	
\end{equation}
Furthermore, for every $a\in\mathcal{A}^1$ and every  non-zero real number $\alpha$, we have 
\begin{equation}\label{Phi_alpha}
	\Phi_\alpha(a)=(1+\alpha^{2})a(1+\alpha^2a)^{-1} \quad \mbox{and} \quad \Phi_\alpha^{-1}(a)=a \left(1+\alpha^2(1-a)\right)^{-1}.	
\end{equation}

The following theorem, motivated by \cite[Theorem 2.3]{Semrl}, shows that,  relying on the formulas \eqref{Phi_T} and \eqref {inverse_Phi_T}, one can describe a large class of order isomorphisms between effect algebras of   $AW^*$-algebras.

\begin{thm}\label{general_formula}
Let $\mathcal{A}$ and $\mathcal{B}$ be two $AW^*$-algebras. Assume that $\mathcal{A}$ does not have a type~$\mathrm{I}_1$ direct summand and  $\Phi:\mathcal{A}^1\to \mathcal{B}^1$  is a map. The following assertions are equivalent:

\begin{enumerate}
	\item  The map $\Phi$ is an order isomorphism with $\Phi(1/2)$ and $1-\Phi(1/2)$ are invertible in $\mathcal{B}$.
	\item  There exist a  Jordan $^*$-isomorphism $J:\mathcal{A}\to\mathcal{B}$,      $T\in\mathcal{B}^{++}$,  and a positive  real number  $\alpha$ such that $$\Phi=\Phi_\alpha^{-1}\circ\Phi_{T}\circ J\arrowvert_{\mathcal{A}^1}.$$
\end{enumerate} 
Moreover, if  $\mathcal{A}$ and $\mathcal{B}$ are finite, then $\Phi$ is an order isomorphism if, and only if, there exist a  Jordan $^*$-isomorphism $J:\mathcal{A}\to\mathcal{B}$, and two invertible closed operators $T\in\mathcal{R(A)}$, $S\in\mathcal{R(B)}$ such that $$\Phi=\Phi_{S}^{-1}\circ J \circ\Phi_{T}.$$
\end{thm}

\begin{proof}
	$(\textit{2})\Rightarrow(\textit{1}):$ It is easy to see that $\Phi$  is an order isomorphism. By \eqref{Phi(alpha)}, $\Phi_T(1/2)=(1+T^2)(2+T^2)$. Using \eqref{Phi_alpha}, by a simple computation we obtain that $\Phi\left(1/2\right) = 1 - \left(1+(1+\alpha^2)^{-1}(T^2 + 1)\right)^{-1}$. Since $T\in\mathcal{B}^{++}$, $\Phi(1/2)$ and $1-\Phi(1/2)$ are also in $\mathcal{B}^{++}$.

$(\textit{1})\Rightarrow(\textit{2}):$ Since  $\Phi(1/2)$ and $1-\Phi(1/2)$ are invertible in $\mathcal{B}$, then there exists an $\varepsilon\in(0,1/2)$ such that $\varepsilon<\Phi(1/2)<1-\varepsilon$. Denote $\alpha:= \left( \frac{1-2\varepsilon}{\varepsilon}\right)^{1/2}$. A simple computation shows that $1/2<\Phi_\alpha\left( \Phi(1/2)\right)<\Phi(1-\varepsilon)<1$. Thus $T:=\left(2\Phi_\alpha(\Phi(1/2))-1\right)^{1/2}\left(1- \Phi_\alpha(\Phi(1/2))\right) ^{-1/2}$ is a well defined  positive invertible element in $\mathcal{B}$ satisfying 	$$\Phi^{-1}_{T}\left( \Phi_\alpha\left( \Phi(1/2)\right)\right)=1/2,$$ 
this follows from \eqref{Phi(alpha)} and \eqref{Phi_alpha}. Applying Theorem \ref{Jordan}, there exists a Jordan $^*$-isomorphism $J:\mathcal{A}\to\mathcal{B}$ such that $\Phi=\Phi_\alpha^{-1}\circ\Phi_{T}\circ J\arrowvert_{\mathcal{A}^1}.$

We now consider the case where $\mathcal{A}$ and $\mathcal{B}$ are finite $AW^*$-algebras. The if implication is trivial. Conversely,  denote by $p$ the characteristic function in $\mathcal{C(X}):=\{\Phi(1/2)\}''$ of the closure of the open set $\{x\in\mathcal{X}: \Phi(1/2)(x)>1/3\}$. Denote $b_1=1/4p+1/2p^{\perp}\Phi(1/2)$. Since $1/2$ is invertible in $\mathcal{R(A)}$. By Lemma \ref{inverse}, $\Phi(1/2)$ and $1-\Phi(1/2)=\Phi^{\perp}(1/2)$ are both invertible in $\mathcal{R(B)}$, and then $b_1,\Phi(1/2)-b_1$, and $1/2-b_1=1/4p+1/2p^\perp(1-\Phi(1/2))$ are positive invertible in $\mathcal{R(B)}$. Again by Lemma \ref{inverse}, $b_0:=\Phi^{-1}(b_1)$ and $1/2-b_0$ are both positive invertible in  $\mathcal{R(A)}$. Denote $T=(1-2b_0)^{1/2}b_0^{-1/2}$ and $S=(1-2b_1)^{1/2}b_1^{-1/2}$. Depending on \eqref{Phi(alpha)}, an explicit computation yields  that $\Phi_T^{-1}(1/2)=b_0$ and $\Phi_S^{-1}(1/2)=b_1$. So,  $\Phi_S\circ\Phi\circ\Phi_T^{-1}(1/2)=1/2$. By Theorem \ref{Jordan}, there exists a  Jordan $^*$-isomorphism $J:\mathcal{A}\to\mathcal{B}$ such that $\Phi=\Phi_{S}^{-1}\circ J \circ\Phi_{T}.$
\end{proof}

\section{Order isomorphisms between operator intervals}\label{S4}

Building on the structural theorems for order isomorphisms between effect algebras of $AW^*$-algebras established earlier, this section provides proofs of the remaining results that describe order isomorphisms between operator intervals in $AW^*$-algebras, namely, Theorem \ref{intervals}, Proposition \ref{comm_case}, and Theorem \ref{I_1}.

\begin{thm}\label{A+}
Let $\mathcal{A}$ and $\mathcal{B}$ be two $AW^*$-algebras.  Assume that $\mathcal{A}$ does not have a type~$\mathrm{I}_1$ direct summand. A map $\Phi:\mathcal{A}^+\to\mathcal{B}^+$ is an order isomorphism if, and only if, there exist a Jordan $^*$-isomorphism $J:\mathcal{A}\to \mathcal{B}$ and $b\in\mathcal{B}^{++}$ such that $$\Phi(a)=bJ(a)b,\qquad a\in \mathcal{A}^+.$$
\end{thm}
\begin{proof}
Only the necessity part requires proof. Note that the map $\Psi:\mathcal{A}^+\to\mathcal{B}^+$ defined by $$a\mapsto \Phi\left( \|\Phi^{-1}(1)\|+1\right) ^{-1/2}\Phi\left( (\|\Phi^{-1}(1)\|+1)a\right) \Phi\left( \|\Phi^{-1}(1)\|+1\right)^{-1/2},$$
 is an order isomorphism satisfying $\Psi(1)=1$. Denote $n_0=\max\left\lbrace 3, \|\Psi^{-1}(3) \|\right\rbrace $. Let $n\geq n_0$, the elements $T_n:=(n-2)^{1/2}$ and $S_n:=(\Psi(n)-2)^{1/2}$ are  positive invertible elements. By \eqref{Phi(alpha)}, $\Phi_{T_n}^{-1}(1/2)=1/n$ and $\Phi_{S_n}^{-1}(1/2)=\Psi(n)^{-1}$. Notice that the map $\mathcal{A}^1\to\mathcal{B}^1$ defined by $$a\mapsto \Phi_{S_n}\left( \Psi(n)^{-1/2}\Psi(n\Phi_{T_n}^{-1}(a))\Psi(n)^{-1/2}\right)$$  is an order isomorphism that preserves $1/2$. Thus, by Theorem \ref{Jordan}, there exists a Jordan $^*$-isomorphism $J:\mathcal{A}\to\mathcal{B}$ such that 
\begin{equation}\label{Psi_n}
	\Psi(a)=\Psi(n)^{1/2}\left( \Phi_{S_n}^{-1}\circ J\circ\Phi_{T_n} \right)(1/n a) \Psi(n)^{1/2},\qquad a\leq n. 
\end{equation}
In particular, $$\Psi(t)=\Psi(n)^{1/2} \Phi_{S_n}^{-1}\left(J\left( \Phi_{T_n} (t/n )\right) \right)\Psi(n)^{1/2}, \qquad t\in[0,n].$$
Since $T_n$ is a scalar,  $J(\Phi_{T_n} (t/n ))$ is a scalar, then, by \eqref{Phi(alpha)}, we conclude that  for every $t\in[0,n]$ $$\Psi(t)=J(\Phi_{T_n} (t/n ))\left(1-J(\Phi_{T_n} (t/n ))+\Psi(n)^{-1}\left( 2J(\Phi_{T_n} (t/n ))-1\right)  \right)^{-1}.$$
Since this  holds for every $n\geq n_0$,  $\Psi(n)^{-1}$ and $ J(\Phi_{T_n} (t/n ))$ converge (with norm) to $0$ and $t/(t+1)$, respectively, as $n$  increases indefinitely,  for all $t\geq0$. We obtain that $\Psi(t)=t$, for all $t\geq0$. Since $J$ is compatible with the continuous
functional calculus, i.e., for an arbitrary self-adjoint element $a\in\mathcal{A}^{sa}$ and continuous
real function $f$ on the spectrum of $a$, we have $J(f(a)) = f(J(a))$. Using \eqref{Psi_n} and the fact that $T_n=S_n$, for all $n\geq n_0$, we conclude that $\Psi(a)=J(a)$ for all $a\in\mathcal{A}^+$. Putting $b=\left( \Phi\left( \|\Phi^{-1}(1)\|+1\right)( \|\Phi^{-1}(1)\|+1)^{-1} \right)^{1/2} $, the definition of $\Psi$ implies that $\Phi(a)=bJ(a)b,$ for all $a\in\mathcal{A}^+$.
\end{proof}

The proof of the following theorem was partially inspired by  \cite[Proof of Theorem 2.2]{Our1}.
\begin{thm}\label{A++}
Let $\mathcal{A}$ and $\mathcal{B}$ be two $AW^*$-algebras.  Assume that $\mathcal{A}$ does not have a type~$\mathrm{I}_1$ direct summand. A map $\Phi:\mathcal{A}^{++}\to\mathcal{B}^{++}$ is an order isomorphism if, and only if, there exist a Jordan $^*$-isomorphism $J:\mathcal{A}\to \mathcal{B}$ and $b\in\mathcal{B}^{++}$ such that $$\Phi(a)=bJ(a)b,\qquad a\in \mathcal{A}^{++}.$$
\end{thm}
\begin{proof}
Only the necessity part requires proof. Choose an $\varepsilon>0$. The map  $\mathcal{A}^+\to\mathcal{B}^+$  defined by $a\mapsto \Phi(a+\varepsilon)-\Phi(\varepsilon)$ is an order isomorphism. By Theorem \ref{A+}, there exist a Jordan $^*$-isomorphism $J:\mathcal{A}\to\mathcal{B}$ and $b_\varepsilon\in\mathcal{B}^{++}$ such that $$\Phi(a)=b_{\varepsilon}J_{\varepsilon}(a)b_{\varepsilon}+c_{\varepsilon}, \qquad a\geq \varepsilon,$$
where $c_{\varepsilon}=\Phi(\varepsilon)-\varepsilon b_{\varepsilon}^2$. Let  $\varepsilon'>0$. For every $t\geq\max\{\varepsilon,\varepsilon'\}$, we have $$1/t\Phi(t)=b_{\varepsilon}^{2}+1/tc_{\varepsilon}=b_{\varepsilon'}^{2}+1/tc_{\varepsilon'}.$$ 
Note that $1/tc_{\varepsilon'}$ and $1/tc_{\varepsilon}$ converge in norm to $0$, as $t$ increases indefinitely. Thus, $b_{\varepsilon}=b_{\varepsilon'}$, this follows from the uniqueness of the square root. We get that $b:=b_{\varepsilon}$ and $c:=c_{\varepsilon'}$ are independent of $\varepsilon$, in other words,    $$\Phi(a)=bJ(a)b+c, \qquad a\in\mathcal{A}^{++}.$$
Note that $\Phi(1/t)=1/tb^2+c\geq 0$, for all $t>0$. Notice that $1/tb^2+c$ converges in norm to $c$, as $t$ increases indefinitely.  Thus, $c\geq 0$. On the other hand, $c\leq bJ(a)b+c=\Phi(a)$, for all $a\in\mathcal{A}^{++}$. The bijectivity of $\Phi$ implies that $c\leq a$, for all $a\in\mathcal{B}^{++}$, which yields that $c\leq 0$. Consequently, $c=0$.  
\end{proof} 

We turn now to the proof of  Theorem \ref{I_1}.

\begin{thm}\label{A^sa}
Let $\mathcal{A}$ and $\mathcal{B}$ be two $AW^*$-algebras.  Assume that $\mathcal{A}$ does not have a type~$\mathrm{I}_1$ direct summand. A map $\Phi:\mathcal{A}^{sa}\to\mathcal{B}^{sa}$ is an order isomorphism if, and only if, there exist a Jordan $^*$-isomorphism $J:\mathcal{A}\to \mathcal{B}$, $b\in\mathcal{B}^{++}$ and $c\in\mathcal{B}^{sa}$ such that $$\Phi(a)=bJ(a)b+c,\qquad a\in \mathcal{A}^{sa}.$$
\end{thm}
\begin{proof}
Only the necessity part requires proof. The map $\mathcal{A}^+\to \mathcal{B}^+$ defined by $a\mapsto -\Phi(-a)+\Phi(0)$ is an order isomorphism. By Theorem \ref{A+}, there exist a Jordan $^*$-isomorphism $J:\mathcal{A}\mapsto\mathcal{B}$ and $b\in\mathcal{B}^{++}$ such that 
$$\Phi(a)=bJ(a)b+ \Phi(0), \qquad a\leq 0.$$	
Let $\Psi:\mathcal{A}^{sa}\to\mathcal{B}^{sa}$ be the map , defined by $a\mapsto b^{-1}\left( \Phi(a)-\Phi(0)\right) b^{-1}$. It is easy to see that  $\Psi$ is an order isomorphism satisfying $\Psi(a)=J(a)$, for all $a\leq0$. Let $n$ be a positive integer,  since the	 map $\mathcal{A}^+\to \mathcal{B}^+$ defined by $a\mapsto \Psi(a-n)+n$ is an order isomorphism, then there exist a Jordan $^*$-isomorphism $J_n:\mathcal{A}\mapsto\mathcal{B}$ and $b_n\in\mathcal{B}^{++}$ such that $$\Psi(a)=b_nJ_n(a)b_n+ n(b_n^2-1), \qquad a\geq -n.$$
Note that $\Psi(0)=J(0)=n(b_n^2-1)=0$, and then $b_n=1$. Hence,  $\Psi(a)=J_n(a)$, for all $a\geq -n$. The fact that $J=J_n$ follows from the equality $\Psi\arrowvert_{[-n,0]_{\mathcal{A}}}= J\arrowvert_{[-n,0]_{\mathcal{A}}}=J_n\arrowvert_{[-n,0]_{\mathcal{A}}}$. The latter
holds for every $n>0$, which means that $\Psi=J\arrowvert_{\mathcal{A}^{sa}}$.  Therefore, from  the definition of $\Psi$, we obtain that $$\Phi(a)=bJ(a)b+c,\qquad a\in \mathcal{A}^{sa},$$ where $c:=\Phi(0)$. This completes the proof.	
\end{proof}

Now, we turn to the proof of Theorem \ref{I_1}.
\begin{proof}[Proof of Theorem \ref{I_1}:]
	We first consider  the case where $\mathcal{I}=\mathcal{A}^1$ and 	$\mathcal{J}=\mathcal{B}^1$.  Let  $p\in\mathcal{P(A)}$.   recall that  $p$ is central if, and only if, it has a unique complement, i.e., there exists a unique projection  $p'\in\mathcal{P(A)}$ such that $\mathrm{inf}_{\mathcal{P(A)}}\{p,p'\}=0$ and $\mathrm{sup}_{\mathcal{P(A)}}\{p,p'\}=1$. Since $p_{\mathcal{A}}(1)$ is the greatest central  projection in  $\mathcal{P(A)}$, such that all projection $p\leq p_{\mathcal{A}}(1)$ are  also central, we conclude that  $\Phi\left( p_\mathcal{A} (1)\right) =p_\mathcal{B} (1)$ and $\Phi\left( p_\mathcal{A} (1)^\perp\right)= p_\mathcal{B} (1)^\perp$. Thus, the restriction $\Phi_i:\Phi\arrowvert_{\mathcal{A}_i^1}$,  $i\in\{1,-1\}$,   is an order isomorphism onto   $\mathcal{B}_i^1$. Let $a\in\mathcal{A}^1$, $\Phi(a)=p_\mathcal{B} (1)\Phi(a)+p_\mathcal{B} (1)^\perp\Phi(a)$. Since $p_\mathcal{A} (1)a=\mathrm{sup}_{\mathcal{A}^{1}}\{a,p_\mathcal{A} (1)\}$, we obtain that $$\Phi_1\left( p_\mathcal{A} (1)a\right)  =\Phi\left( p_\mathcal{A} (1)a\right)=\mathrm{sup}_{\mathcal{B}^1}\{\Phi(a),p_\mathcal{B} (1)\}=p_\mathcal{B} (1)\Phi(a).$$ 
	Similarly, we may show that $\Phi_{-1}\left( p_\mathcal{A} (1)^\perp a\right)=p_\mathcal{B} (1)^\perp\Phi(a)$. Thus $\Phi(a)=\Phi_{1}(p_\mathcal{A} (1)a)+\Phi_{-1}(p_\mathcal{A} (1)^{\perp}a)$. 
	
	For the general case. Let  $a_1,a_{-1}\in\mathcal{I}$ such that $a_{-1}<a_{1}$,  and $\Phi(a_{-1})<\Phi(a_1)$.
	Let $\Gamma:[a_{-1},a_1]_\mathcal{A}\to\mathcal{A}^1$ and $\Theta:\mathcal{B}^1 \to[\Phi(a_{-1}),\Phi(a_1)]_{\mathcal{B}}$ be the order isomorphisms defined by 
	$$\Gamma(a):= (a_1-a_{-1})^{-1/2}(a-a_{-1})(a_1-a_{-1})^{-1/2},$$
	and 
	$$\Theta(b):=\left( \Phi(a_1)-\Phi(a_{-1})\right) ^{1/2}b\left( \Phi(a_1)-\Phi(a_{-1})\right)^{1/2} +\Phi(a_{-1}).$$
	Define $\Gamma_1: \mathcal{P}_\mathcal{A} (1)[a_{-1},a_{1}]_\mathcal{A}\to \mathcal{P}_\mathcal{A} (1)\mathcal{A}^1$   by 
	$$\Gamma_1\left(\mathcal{P}_\mathcal{A} (1)a\right) :=\mathcal{P}_\mathcal{A} (1)\Gamma(a),  \qquad a\in [a_{-1},a_1]_\mathcal{A}.$$
	Define   $\Theta_1:\mathcal{P}_\mathcal{B} (1)\mathcal{B}^1 \to  \mathcal{P}_\mathcal{B} (1)[\Phi(a_{-1}),\Phi(a_{1})]_\mathcal{B}$ by 
	$$\Theta_1\left(\mathcal{P}_\mathcal{B} (1)b\right) :=\mathcal{P}_\mathcal{B} (1)\Theta(b),  \qquad b\in \mathcal{B}^1.$$
	Note that $\Gamma_1, \Theta_1$ are  well defined order isomorphisms. Using $p_\mathcal{A} (1)^{\perp}$ and $p_\mathcal{B} (1)^{\perp}$  instead of $p_\mathcal{A}(1)$ and $p_\mathcal{A}(1)$, we may define similarly $\Gamma_{-1}$ and $\Theta_{-1}$. On the other hand  the map $\Psi:=\Theta^{-1}\circ\Phi\circ \Gamma^{-1}$  is  order isomorphism between $\mathcal{A}^1$  and $\mathcal{B}^1$.  By a discussion similar to that in the preceding
	paragraph, we can show that there exist two order isomorphisms  $\Psi_1:p_\mathcal{A} (1)\mathcal{A}^1\to p_\mathcal{B}(1)\mathcal{B}^1$  and $\Psi_{-1}:p_\mathcal{A} (1)^\perp\mathcal{A}^1\to p_\mathcal{B} (1)^\perp\mathcal{B}^1$ such that $p_\mathcal{B} (1)\Psi(a)=\Psi_1\left( p_\mathcal{B} (1)a\right)$ and $p_\mathcal{B} (1)^\perp\Psi(a)=\Psi\left( p_\mathcal{B} (1)^\perp a\right)$, for all $a\in\mathcal{A}^{1}$. Thus, $\Phi_{(a_{-1},a_1)}^{(1)}:=\Theta_1\circ \Psi_1 \circ \Gamma_1$  and $\Phi_{(a_{-1},a_1)}^{(-1)}:=\Theta_{-1}\circ \Psi_{-1} \circ \Gamma_{-1}$ are order isomorphisms that satisfy, for every $a\in[a_{-1},a_1]_{\mathcal{A}}$,
	\begin{equation}\label{independence}
		\Phi_{(a_{-1},a_1)}^{(1)}\left( p_\mathcal{A} (1)a\right)=p_\mathcal{B} (1)\Phi\left(a\right)\ \mbox{and} \   \Phi_{(a_{-1},a_1)}^{(-1)}\left( p_\mathcal{A} (1)^\perp a\right)=p_\mathcal{B} (1)^\perp\Phi\left( a\right).	
	\end{equation}  
	
	Note that for each pair $(a,b) \in \mathcal{I}^2$, there exists a pair $(a_{-1},a_1) \in \mathcal{I}^2$ such that $a_{-1} < a_1$, $\Phi(a_{-1}) < \Phi(a_1)$, and $a,b \in [a_{-1},a_1]_{\mathcal{A}}$. Using the previous fact together with \eqref{independence}, it is easily verified that the maps $\Phi_1 \colon p_\mathcal{A}(1)\mathcal{I} \to p_\mathcal{B}(1)\mathcal{J}$ and $\Phi_{-1} \colon p_\mathcal{A}(1)^\perp\mathcal{I} \to p_\mathcal{B}(1)^\perp\mathcal{J}$, defined by  
	$$
	p_\mathcal{A}(1)a \mapsto p_\mathcal{B}(1)\Phi(a) \quad \text{and} \quad p_\mathcal{A}(1)^\perp a \mapsto p_\mathcal{B}(1)^\perp\Phi(a),
	$$ 
	respectively, are well defined order isomorphisms. This completes the proof of the theorem.
\end{proof}

\begin{thm}\label{++non=sa}
Let $\mathcal{A}$ and $\mathcal{B}$ be two $AW^*$-algebras. Then, $\mathcal{A}^{sa}$  and $\mathcal{B}^{++}$ are order isomorphic if, and only if, $\mathcal{A}$ and $\mathcal{B}$  are Jordan $^*$-isomorphic and $\mathcal{A}$ is commutative.
\end{thm}
\begin{proof}
To establish the sufficiency  direction, suppose $J: \mathcal{A} \to \mathcal{B}$ is a Jordan $^*$-isomorphism. We define the map $\Phi: \mathcal{A}^{\text{sa}} \to \mathcal{B}^{++}$ by $\Phi(a) = J(\exp(a))$, where $\exp$ denotes the  exponential  function. This map induces an order isomorphism between $\mathcal{A}^{sa}$ and $\mathcal{B}^{++}$.

Conversely, suppose that   $\mathcal{A}$  is not  commutative ($\mathcal{A}_{-1}\neq\{0\}$), and  $\mathcal{A}^{sa}, \mathcal{B}^{++}$  are  order isomorphic. By Theorem \ref{I_1}, we can suppose without loss of generality that $\mathcal{A}$  and $\mathcal{B}$  do not have $\mathrm{I}_1$ direct summand. Let  $\Phi:\mathcal{A}^{sa}\to\mathcal{B}^{++}$  be an order isomorphism. For every $n>0$, the map $\mathcal{A}^{+}\to\mathcal{B}^{+}$, defined by $a\mapsto \Phi(a-n)-\Phi(-n)$, is an order isomorphism. By Theorem \ref{A+}, there exist a Jordan $^*$-isomorphism  $J_n:\mathcal{A}\to\mathcal{B}$ and $b_n\in\mathcal{B}^{++}$ such that 
\begin{equation}\label{Phi(0)}
	\Phi(a)=b_nJ_n(a)b_n+ nb_n^2+\Phi(-n),\qquad a\geq -n.\end{equation}
Let $m,n>0$, $\Phi(0)=nb_n^2+\Phi(-n)=mb_m^2+\Phi(-m)$. Thus $\Phi(1)=b_n^2+\Phi(0)=b_m^{2}+\Phi(0)$, which implies that $b:=b_n=b_m$ is dependent of  the choice of $n$.  By \eqref{Phi(0)}, 
\begin{equation}\label{Phi(00)}
	\|\Phi(0)\|\geq n\|b\|^2-\|\Phi(-n)\|, \qquad n>0.
\end{equation}
Since $\Phi(-n)\leq \Phi(0)$, for all $n>0$, we conclude that the sequence  $(\|\Phi(-n)\|)$ is bounded.  The combination of this with \eqref{Phi(00)} leads to a contradiction.
\end{proof}

We now proceed to prove Proposition \ref{comm_case}, which addresses the commutative case. Recall that in \cite[Theorem 1]{sanchez} and \cite[Theorem 2.1]{Ehsani}, the authors provide proofs for the self-adjoint parts and effect algebras, whereas \cite{Marovt} examines positive cones under the restrictive assumption that the compact space is first-countable. To streamline our proof, we will be content to focus only on the cone case.
\begin{proof}[Proof of Proposition \ref{comm_case}]
	Let $\Phi: \mathcal{C}(\mathcal{X}, \mathbb{R}^+) \to \mathcal{C}(\mathcal{Y}, \mathbb{R}^+)$ be an order isomorphism. Every element $f \in  \mathcal{C}(\mathcal{X}, \mathbb{R})$ can be written uniquely as $f = f^+ - f^-$ for some $f^+, f^- \in \mathcal{C}(\mathcal{X}, \mathbb{R}^+)$ satisfying $f^+f^- = 0$. This ensures that the map  $\Psi: \mathcal{C}(\mathcal{X}, \mathbb{R}) \to \mathcal{C}(\mathcal{Y}, \mathbb{R})$ given by  
	$$
	\Psi(f^+ - f^-) = \Phi(f^+) - \Phi(f^-), \quad f^+, f^- \in \mathcal{C}(\mathcal{X}, \mathbb{R}^+) \text{ with } f^+f^- = 0,
	$$
	is well defined.  The bijectivity of $\Psi$ follows from the bijectivity of $\Phi$ and the fact that for any $f^+, f^- \in \mathcal{C}( \mathcal{X}, \mathbb{R}^+)$, we have  
	\begin{center}
		$f^+f^- = 0 \quad \text{if and only if} \quad \inf_{\mathcal{C}( \mathcal{X}, \mathbb{R}^+)}\{f^+, f^-\} = 0.$	
	\end{center}
	For $f, g \in \mathcal{C}(\mathcal{X}, \mathbb{R})$, the inequality $f \leq g$ holds if and only if $f^+ \leq g^+$ and $f^- \geq g^-$. Since $\Phi$ is an order isomorphism, $\Psi$ preserves the order in both directions. Thus, $\Psi$ is also an order isomorphism.  
	
	By \cite[Theorem 1]{sanchez}, there exist a homeomorphism $\mu: \mathcal{Y} \to \mathcal{X}$, a dense $G_\delta$-subset $\mathcal{Y}_0$  of $\mathcal{Y}$, and a family of increasing homeomorphisms $\{f_y\}_{y \in \mathcal{Y}_0}$ of $\mathbb{R}^+$ such that  
	$$
	\Phi(f)(y) = f_y\left(f(\mu(y))\right), \quad \text{for all } f \in \mathcal{C}(\mathcal{X}, \mathbb{R}^+), \, y \in \mathcal{Y}_0.
	$$
	This completes the proof.
\end{proof}

\section{ The reduction of the problem}\label{S5}

The set of self-adjoint elements in a $C^*$-algebra $\mathcal{A}$ forms a function system, that is, an Archimedean partially ordered real vector space where the algebra’s unit serves as the order unit; see \cite[pp. 588–589]{Choi}. Indeed,  by spectral theory, for each  $a\in\mathcal{A}^{sa}$,  we have
\begin{center}
	$\|a\|=\inf_{\mathbb{R}}\{t: -t\leq a \leq t\}.$
\end{center}
This implies  that $1$  serves as an order unit in $\mathcal{A}^{sa}$. Furthermore, it shows that the norm of  $\|.\|$  coincides with the order-unit norm. Importantly, the order structure of a $C^*$-algebra depends entirely on the order structure of the vector space formed by its self-adjoint elements. This observation motivates the utility of linearly embedding such vector spaces into larger function systems with specific properties  while preserving their original order structure.

We begin by defining the concept  of the regularity. In this context, it is an analogue of  the concept of topological density to the order-theoretic setting, where approximations are defined via suprema rather than topological limits.  
\begin{defn}
	Let $\mathcal{W}$ be   a function system  and $\mathcal{V}$ a subspace. For $w\in\mathcal{W}$, we define $$(-\infty, w]_{\mathcal{V}}:=\{v\in\mathcal{V}: v\leq w\} \qquad \mbox{and} \qquad [w,\infty)_{\mathcal{V}}:=\{v\in\mathcal{V}: v\geq w\}.$$ 
	We say that $\mathcal{V}$ is a regular subspace of $\mathcal{W}$ if, for each $w\in\mathcal{W}$  \begin{center}
		$w=\sup_{\mathcal{W}}(-\infty, w]_{\mathcal{V}}.$
	\end{center}
	If $\mathcal{V}$ is a regular subspace of $\mathcal{W}$, then a straightforward verification shows that  for every  $w\in\mathcal{W}$  
	\begin{center}
		$w=\inf_{\mathcal{W}}[w,\infty)_{\mathcal{V}}.$
	\end{center}
\end{defn}
We observe that this definition aligns with the one provided in Subsection \ref{2.2}, and furthermore,   $\mathcal{A}^{sa}$ is a regular subspace of $\overline{\mathcal{A}}^{sa}$.

\begin{defn}
	Let $\mathcal{V}$ be a function system. An extension of $\mathcal{V}$  is a pair $(\mathcal{W}, j)$, where $\mathcal{W}$ is another  function system, and $j:\mathcal{V}\to \mathcal{W}$ is a  unital linear order injection, that is, an order isomorphism onto $j(\mathcal{V})\subseteq\mathcal{W}$ with $j(1)=1$. The extension $(\mathcal{W}, j)$ is regular if $j(\mathcal{V})$ is a regular subspace of $\mathcal{W}$.  An order  injection $j:\mathcal{V}\to\mathcal{W}$ between function systems is said to be sup-preserving if $\sup_{\mathcal{V}}\mathcal{F}=v$ with $\mathcal{F}$
	a subset of $\mathcal{V}$, then $\sup_{\mathcal{W}}j\mathcal{(F)}=j(v)$.   
\end{defn}
The following  lemma is quoted from  \cite[Lemma 2.4]{Hamana}.
\begin{lem}\label{regular-sup-preserving}
	If $(\mathcal{W}, j)$ is a regular extension of $\mathcal{V}$, then $j$ is sup-preserving.	
\end{lem}

A Dedekind completion of a function system $\mathcal{V}$ is a regular extension  $( \mathcal{W},j)$, where $\mathcal{W}$ is a boundedly complete vector lattice, i.e.  a lattice in which every subset that is bounded above has a least upper-bound and every subset that is bounded below has a greatest lower-bound. In \cite[Lemma 1.2, Theorem 2.9]{Wright}, Wright proved that any function system has a unique Dedekind completion up to  linear order isomorphism. In particular, for each (unital) $C^*$-algebra $\mathcal{A}$, there exists a Dedekind completion $\left( \mathcal{V}(\mathcal{A}),j\right)$ for the function system $\mathcal{A}^{sa}$. 

Let $\mathcal{A}$ be a $C^*$-algebra and $\overline{\mathcal{A}}$ its regular monotone completion. Let  $\left( \mathcal{V}(\overline{\mathcal{A}}),j\right)$ be the Dedekind completion for the function system $\overline{\mathcal{A}}^{sa}$. From  Lemma \ref{regular-sup-preserving}, we know that  $j$ is sup-preserving. Hence, we can identify $\overline{\mathcal{A}}^{sa}$ with $j(\overline{\mathcal{A}}^{sa})$ without loss of any information on the order structure of $\overline{\mathcal{A}}^{sa}$. More precisely, if  $\sup_{\overline{\mathcal{A}}^{sa}}\mathcal{F}=v $ with $\mathcal{F}$
a subset of $\overline{\mathcal{A}}^{sa}$, then we have  $\sup_{\mathcal{V}(\overline{\mathcal{A}})}j(\mathcal{F})=j(v)$. After this identification,  $\mathcal{A}^{sa}$ can also be regarded   as a subspace  of $\mathcal{V}(\overline{\mathcal{A}})$. 

The following proposition  follows  immediately from Lemma \ref{regular-sup-preserving} and \cite[Lemma 2.5]{Hamana}.
\begin{prop}\label{A-regular}
	Let $\mathcal{A}$	be a $C^*$-algebra. Then  $\mathcal{A}^{sa}$ is  a regular  subspace of $ \mathcal{V}(\overline{\mathcal{A}})$. Furthermore, the embedding $\mathcal{A}^{sa} \hookrightarrow \mathcal{V}(\overline{\mathcal{A}})$ is sup-preserving.
\end{prop}

We are now in a position to prove Theorem \ref{extension}, which reduces Problem \ref{problem1}, and more generally, the problem of studying order isomorphisms between the self-adjoint parts of $C^*$-algebras, to the case of monotone complete $C^*$-algebras.

\begin{proof}[Proof of Theorem \ref{extension}:]
	First we prove that $\Phi$ is extendible to a unique order isomorphism $\Phi_\mathcal{V}:\mathcal{V}(\overline{\mathcal{A}})\to \mathcal{V}(\overline{\mathcal{B}})$. Let $v\in\mathcal{V}(\overline{\mathcal{A}})$, then $\Phi((-\infty,v]_{\mathcal{A}^{sa}})$ is bounded by  $\Phi(\|v\|)$, that is, $w\leq\Phi(\|v\|)$ for all $w\in \Phi((-\infty,v]_{\mathcal{A}^{sa}})$; we write $\Phi((-\infty,v]_{\mathcal{A}^{sa}})\leq \Phi(\|v\|)$. Since $\mathcal{V}(\overline{\mathcal{B}})$ is boundedly complete, we obtain that \begin{center}
		$\Phi_\mathcal{V}(v):=\sup_{\mathcal{V}(\overline{\mathcal{B}})}\Phi((-\infty,v]_{\mathcal{A}^{sa}})\in \mathcal{V}(\overline{\mathcal{B}})$.
	\end{center}

We will prove that  \begin{equation}\label{i}
		\Phi([v,\infty)_{\mathcal{A}^{sa}})=[\Phi_\mathcal{V}(v),\infty)_{\mathcal{B}^{sa}}.   
	\end{equation}
	Indeed, notice that $[v,\infty)_{\mathcal{A}^{sa}}= \{w\in \mathcal{A}^{sa}: w\geq (-\infty,v]_{\mathcal{A}^{sa}}\}$.  Then $$\begin{array}{lllll}
		\Phi([v,\infty)_{\mathcal{A}^{sa}})&=&  \{w\in \mathcal{B}^{sa}: w\geq \Phi((-\infty,v]_{\mathcal{A}^{sa}})\}\\ \\
		&=& \{w\in \mathcal{B}^{sa}: w\geq \Phi_\mathcal{V}(v)\}\\ \\ 
		&=& [\Phi_\mathcal{V}(v),\infty)_{\mathcal{B}^{sa}}.
	\end{array}
	$$
	
	Using \eqref{i}, a  similar reasoning  shows that \begin{equation}\label{ii}
		\Phi((-\infty,v]_{\mathcal{A}^{sa}})=(-\infty,\Phi_\mathcal{V}(v)]_{\mathcal{B}^{sa}}.
	\end{equation}  
	
	The map $\Phi_{\mathcal{V}}:\mathcal{V}(\overline{\mathcal{A}})\to \mathcal{V}(\overline{\mathcal{B}})$ given  by $v\mapsto \Phi_{\mathcal{V}}(v)$ is well defined. Now we will prove the bijectivity of $\Phi$. Using Proposition \ref{A-regular}, we deduce the injectivity of $\Phi$ directly from  \eqref{i}.  Let $w\in  \mathcal{V}(\overline{\mathcal{B}})$ and  $v=\sup_{\mathcal{V}(\overline{\mathcal{A}})}\Phi^{-1}((-\infty,w]_{\mathcal{B}^{sa}})$. Since $\Phi$ and $\Phi^{-1}$ have the same properties, using \eqref{ii} we obtain that $$\Phi^{-1}((-\infty,w]_{\mathcal{B}^{sa}})=(-\infty,v]_{\mathcal{A}^{sa}},$$
	or equivalently, $\Phi_\mathcal{V}(v)=w$. This shows that $\Phi_\mathcal{V}$ is surjective, and then $\Phi_\mathcal{V}$ is bijective. The fact that $\Phi_\mathcal{V}$ preserves the order in both directions comes immediately from \eqref{ii} and  the equivalence that $v\leq v'$ if, and only if $(-\infty,v]_{\mathcal{A}^{sa}}\subseteq (-\infty,v']_{\mathcal{A}^{sa}}$, for all $v,v'\in \mathcal{V}(\overline{\mathcal{A}})$. Consequently, $\Phi_\mathcal{V}$ is an order isomorphism satisfying $\Phi_\mathcal{V}(v)=\Phi(v)$ for all $v\in\mathcal{A}^{sa}$. 
	
	Since the monotone closure of $\mathcal{A}^{sa}$ in $\mathcal{V}(\overline{\mathcal{A}})$ is the intersection of all monotone closed subset of $\mathcal{V}(\overline{\mathcal{A}})$ which contain $\mathcal{A}^{sa}$. Then, $\Phi_\mathcal{V}$ maps  the monotone closure of $\mathcal{A}^{sa}$ onto that  of $\mathcal{B}^{sa}$.  By Lemma \ref{regular-sup-preserving},  the monotone closure of $\mathcal{A}^{sa}$  in $\mathcal{V}(\overline{\mathcal{A}})$  is equal to the   monotone closure of $\mathcal{A}^{sa}$ in $\overline{\mathcal{A}}^{sa}$, which is $\overline{\mathcal{A}}^{sa}$. Then, $\Phi_\mathcal{V}(\overline{\mathcal{A}}^{sa})=\overline{\mathcal{B}}^{sa}$. Thus, to complete the proof, we denote the restriction of $\Phi_\mathcal{V}$ to $\overline{\mathcal{A}}^{sa}$ by $\overline{\Phi}$.
\end{proof}

\section{Remarks and intriguing questions}\label{S6}

Our results achieve a significant level of generality, encompassing all prior studies in this field. Additionally, the description encompassing all operator intervals in  $AW^*$-algebras is optimal and offers a comprehensive description, with the sole exception of effect algebras.   Notably, the formulas presented in Theorem \ref{general_formula} are not unique.

\begin{rem}
Using a similar approach as in \cite[Theorem 2.2]{Our}, we can prove that the conditions $\textit{(1)}$  and  $\textit{(2)}$ in Theorem \ref{general_formula} are equivalent to:	 $$\Phi(a)=T\left( a(T^2-1)+1\right) ^{-1}aT, \qquad a\in\mathcal{A}^1,$$
where $T=(\Phi(1/2)^{-1}-1)^{-1/2}$.
\end{rem}

For every real
number $\alpha < 1$, we define the function $f_{\alpha} : [0, 1] \to[0, 1]$ by
$$f_{\alpha}(t) = \dfrac{t}{t\alpha+1-\alpha}, \qquad t\in[0, 1].$$
\begin{rem}
Paralleling theorems 2.2 and 2.3 of~\cite{Semrl} (while retaining the notation of Theorem~\ref{general_formula} and using similar arguments from its proof), we can show that if $\mathcal{A}$ and $\mathcal{B}$ are finite $AW^*$-algebras, then the following conditions are equivalent:

\begin{enumerate}
	\item $\Phi$ is an order isomorphism.
	\item There exist a  Jordan $^*$-isomorphism $J:\mathcal{A}\to\mathcal{B}$,    an invertible element $T\in\mathcal{R(B)}$ and two real numbers $ 0 < \alpha < 1, \beta < 0$ such that $$\Phi(a)=f_{\beta}\left( (f_{\alpha}(TT^*))^{-1/2} f_{\alpha}(TJ(a)T^*) (f_{\alpha}(TT^*))^{-1/2}\right), \qquad a\in\mathcal{A}^{1}.$$
\end{enumerate} 	
\end{rem}	

In the finite case, we have a complete description of all order isomorphisms between effect algebras. However, the  properly infinite case remains open. We believe that studying the preservation of invertibility under order isomorphisms is a key approach in this case. We observe that order isomorphisms between effect algebras preserve elements with the unity as their range projection, which are essentially the injective operators (if we consider the elements of $C^*$-algebra as operators acting on a Hilbert space). Consequently, in general, invertibility can only be preserved through these isomorphisms is only achievable if we are willing to work with  unbounded operators or with regular Baer rings. 

If an extension analogous to Theorem~\ref{extension} holds for other operator intervals, Theorem~\ref{intervals}  
immediately yields the general forms of all order isomorphisms between such intervals in arbitrary $C^*$-algebras.   This methodology, being both efficient and optimal, constitutes the most streamlined approach to the problem to date.

To conclude this work, we formulate the following questions, motivated by   Theorem \ref{extension},
and  Theorem \ref{general_formula}.  
An affirmative answer to these questions would enable a complete characterization  
of all order isomorphisms between operator intervals within general $C^*$-algebras.  

\begin{prob}\label{final-questions}  
	\begin{itemize}  
		\item Let $\mathcal{A}$ be a $C^*$-algebra. Is the monotone closure of $\mathcal{A}^1$ in $\overline{\mathcal{A}}^1$ equal to $\overline{\mathcal{A}}^1$ itself?  
		
		\item Let $\mathcal{A}$ be a properly infinite $AW^*$-algebra. Does there exist  
		an embedding of $\mathcal{A}$ into a regular Baer $^*$-ring?  
	\end{itemize}  
\end{prob}

\section*{Acknowledgement}
The author expresses his sincere gratitude to his advisor, Professor Zine El Abidine Abdelali, from the Department of Mathematics, Faculty of Sciences, Mohamed V University in Rabat, for his valuable guidance and constructive suggestions, which have greatly improved the quality of this manuscript.

\section*{Disclosure statement}
No potential conflict of interest was reported by the author.

\section*{Funding }
The author has no relevant financial or non-financial interests to disclose.

\section*{Data availability statement }

Data sharing not applicable to this article as no datasets  were generated  or analysed during the current study.

\end{document}